\documentclass[11pt]{article}
\usepackage{amsfonts,amscd,amssymb, amsmath}

\allowdisplaybreaks[4]
\newtheorem{definition}{Definition}[section]

\newtheorem{proposition}[definition]{Proposition}
\newtheorem{corollary}[definition]{Corollary}
\newtheorem{remark}[definition]{Remark}

\newtheorem{theorem}[definition]{Theorem}
\newtheorem{example}[definition]{Example}

\newcommand{\nat}{\mbox{$\;\natural \;$}}

\def\rawo\lonra{\longrightarrow}

\def\ot{\otimes}

\newcommand{\selabel}[1]{\label{se:#1}}

\def\hot{\hat{\otimes }}
\def\le{\langle}
\def\ri{\rangle}

\allowdisplaybreaks[4]

\newenvironment{proof}{{\it Proof.}}{\hfill $ \square $ \vskip 4mm}

\begin{document}
\title{Hom-L-R-smash products, Hom-diagonal crossed products and the Drinfeld double 
of a Hom-Hopf algebra}
\author{Abdenacer Makhlouf\\
Universit\'{e} de Haute Alsace, \\
Laboratoire de Math\'{e}matiques, Informatique et Applications, \\
4, Rue des Fr\`{e}res Lumi\`{e}re, F-68093 Mulhouse, France\\
e-mail: Abdenacer.Makhlouf@uha.fr
\and Florin Panaite\thanks {Work supported by a grant of the Romanian National 
Authority for Scientific Research, CNCS-UEFISCDI, 
project number PN-II-ID-PCE-2011-3-0635,  
contract nr. 253/5.10.2011.}\\
Institute of Mathematics of the
Romanian Academy\\
PO-Box 1-764, RO-014700 Bucharest, Romania\\
 e-mail: Florin.Panaite@imar.ro}
\date{}
\maketitle

\begin{abstract}
We  introduce the Hom-analogue of the L-R-smash product and use it to define 
the Hom-analogue of the diagonal crossed product. When $H$ is a finite dimensional Hom-Hopf algebra 
with bijective antipode and bijective structure map, we define the Drinfeld double of $H$; its algebra 
structure is a Hom-diagonal crossed product and it has all expected properties, namely it is quasitriangular and 
modules over it coincide with left-right Yetter-Drinfeld modules over $H$. 
\end{abstract}
\section*{Introduction}
${\;\;\;}$
Hom-type algebras appeared first in physical contexts, in connection with twisted, discretized or deformed 
derivatives and corresponding generalizations, discretizations and deformations of vector fields and differential 
calculus (see \cite{Alvarez,AizawaSaito,ChaiElinPop,ChaiKuLukPopPresn,ChaiIsKuLuk,ChaiPopPres,CurtrZachos1,
DamKu,DaskaloyannisGendefVir,Hu,Kassel1,LiuKeQin}). These papers dealt mainly with $q$-deformations of  Heisenberg algebras (oscillator algebras),  the Virasoro algebra and quantum conformal algebras, applied in Physics within 
string theory, vertex operator models, quantum
scattering, lattice models and other contexts.

In \cite{HLS,LS1} the authors  showed that a new quasi-deformation scheme leads from Lie algebras to a broader class of quasi-Lie algebras and subclasses of quasi-Hom-Lie algebras and Hom-Lie algebras.  The study of the  class of Hom-Lie algebras, generalizing usual Lie algebras and where the Jacobi identity  is twisted by a  linear map, has  become an active area  of research.  The corresponding associative algebras, called Hom-associative algebras, where introduced in  \cite{ms1} and it was shown that a commutator of a Hom-associative algebra gives rise to a Hom-Lie algebra. For further results  see \cite{AEM,fgs,Gohr,Mak:Almeria,ms2,yau1,yau3,yauhomyb1}. The coalgebra counterpart and the related  notions of 
Hom-bialgebra and Hom-Hopf algebra were introduced in \cite{ms3,ms4}  and 
some of their properties, extending properties of
bialgebras and Hopf algebras, were described. The original definitions of Hom-bialgebra and Hom-Hopf algebra involve two different linear maps $\alpha$ and $\beta$, with $\alpha $ twisting the associativity condition and $\beta $ the coassociativity condition. Afterwards, two directions of study were developed, one considering the class such that 
$\beta=\alpha$, which are still called Hom-bialgebras and Hom-Hopf algebras (cf. \cite{Elhamdadi-Makhlouf,mp1,mp2,yau2,homquantum1,homquantum2,homquantum3}) and another one, initiated in \cite{stef}, 
where the map $\alpha$ is assumed to be invertible and $\beta=\alpha^{-1}$ (these are called monoidal 
Hom-bialgebras and monoidal Hom-Hopf algebras). Yetter-Drinfeld modules, integrals, 
the Drinfeld double and Radford's biproduct  have been studied for monoidal Hom-bialgebras in \cite{CZ,CWZ,LB}. 
Yetter-Drinfeld modules over  Hom-bialgebras were studied  in \cite{mp1} and we will introduce the  
Drinfeld double in this paper. Since Hom-bialgebras and monoidal Hom-bialgebras are different concepts, it turns out 
that our definitions, formulae and results are also different from the ones in \cite{CZ, LB}.

One of the main tools to construct
examples of  Hom-type algebras is the so-called ''twisting principle'' which 
 was introduced by D. Yau for Hom-associative algebras and since then 
extended to various Hom-type algebras. It allows to construct a Hom-type algebra starting from a 
classical-type algebra and an algebra homomorphism.

The twisted tensor product $A\ot _RB$ of two associative algebras $A$ and $B$ is a certain associative algebra 
structure on the vector space $A\ot B$, defined in terms of a so-called twisting map $R:B\ot A\rightarrow 
A\ot B$, having the property that it coincides with the usual tensor product algebra $A\ot B$ if $R$ is the usual flip 
map. This construction was introduced in \cite{Cap,VanDaele} and it may be regarded as a 
representative for the Cartesian product of noncommutative spaces, see \cite{jlpvo, lpvo} for more on this 
subject. In \cite{mp2} we generalized this construction to Hom-associative algebras: if $R$ is a linear map 
$R:B\ot A\rightarrow A\ot B$ between two Hom-associative algebras $A$ and $B$ satisfying some conditions 
(such an $R$ is called Hom-twisting map), we can construct the so-called Hom-twisted tensor product 
$A\ot _RB$, which is a Hom-associative algebra. 

The L-R-smash product over a cocommutative Hopf algebra was introduced 
and studied in a series of papers  
\cite{b1, b2, b3, b4}, inspired by the theory of deformation quantization. This construction was 
substantially generalized in \cite{panvan}, as follows: if $H$ is a bialgebra or quasi-bialgebra and $D$ is an 
$H$-bimodule algebra, the L-R-smash product is a certain associative algebra structure on $D\ot H$, denoted 
by $D\nat H$. A common generalization of twisted tensor products of algebras and L-R-smash products 
over bialgebras was introduced in \cite{cp}, under the name L-R-twisted tensor product of algebras. 

The diagonal crossed product (\cite{bpvo, hn1}) is a construction that associates to a Hopf or quasi-Hopf 
algebra $H$ with bijective antipode and to an $H$-bimodule algebra $D$ a certain associative algebra structure 
on $D\ot H$, denoted by $D\bowtie H$. It was proved in \cite{panvan} that we actually have 
an algebra isomorphism $D\nat H\simeq D\bowtie H$. The importance of the diagonal crossed product 
stems from the fact that, if $H$ is finite dimensional and $D=H^*$, then $H^*\bowtie H$ is the algebra 
structure of the Drinfeld double of $H$. 

The ultimate aim of this paper is to construct the Drinfeld double of a finite dimensional Hom-Hopf algebra $H$ 
with bijective antipode and bijective structure map. We can expect from the beginning that its algebra structure 
has to be a Hom-analogue of a diagonal crossed product between $H^*$ and $H$, but it is not clear at all 
a priori how to define this Hom-analogue. So, we proceed as follows. We define first the Hom-analogue of the 
L-R-twisted tensor product of algebras, which is a natural generalization of the Hom-twisted tensor product. 
It is defined as follows: if $A$ and $B$ are two Hom-associative algebras and $R:B\ot A\rightarrow A\ot B$ and 
$Q:A\ot B\rightarrow A\ot B$ are linear maps satisfying some conditions, the 
Hom-L-R-twisted tensor product $A\;_Q\ot _RB$ is a certain Hom-associative algebra structure on $A\ot B$. 
The key result is Proposition \ref{biject}, saying that if $Q$ is bijective then the map 
$P:=Q^{-1}\circ R:B\ot A\rightarrow A\ot B$ is a Hom-twisting map (and we have an algebra 
isomorphism $A\;_Q\ot _RB\simeq A\ot _PB$). Now if $H$ is a Hom-bialgebra and $D$ is an $H$-bimodule 
Hom-algebra, we can define in a natural way a Hom-L-R-twisted tensor product $D\;_Q\ot _RH$, which is 
denoted by $D\nat H$ and called the Hom-L-R-smash product. It turns out that, under some extra 
hypotheses (among them, the existence of a bijective antipode on $H$), the map $Q$ is bijective, so we have the 
Hom-twisted tensor product $D\ot _PH$, where $P=Q^{-1}\circ R$; this will be the Hom-diagonal crossed 
product $D\bowtie H$ we are looking for. Moreover, if $H$ is a finite dimensional Hom-Hopf algebra 
with bijective antipode and bijective structure map, we can build such a Hom-diagonal crossed product 
$H^*\bowtie H$, and this will be the algebra structure of the Drinfeld double $D(H)$. 

To find the rest of the structure of $D(H)$, we define left-right Yetter-Drinfeld modules over $H$, note 
that they form a braided monoidal category (we analyzed this in detail for left-left Yetter-Drinfeld modules 
in \cite{mp1}), prove that the category of modules over $D(H)$ is isomorphic to the category of 
left-right Yetter-Drinfeld modules and then transfer all the structure from $_H{\mathcal YD}^H$ to 
$D(H)$. It turns out that $D(H)$ is a quasitriangular Hom-Hopf algebra, as expected. 

\section{Preliminaries}\selabel{1}
${\;\;\;}$
We work over a base field $k$. All algebras, linear spaces
etc... will be over $k$; unadorned $\ot $ means $\ot_k$. For a comultiplication 
$\Delta :C\rightarrow C\ot C$ on a vector space $C$ we use a 
Sweedler-type notation $\Delta (c)=c_1\ot c_2$, for $c\in C$. Unless 
otherwise specified, the (co)algebras ((co)associative or not) that will appear 
in what follows are {\em not} supposed to be (co)unital, and a multiplication 
$\mu :V\ot V\rightarrow V$ on a linear space $V$ is denoted by juxtaposition: 
$\mu (v\ot v')=vv'$. 

We recall some concepts and results, fixing the terminology 
to be used throughout the paper. For Hom-structures, we use terminology as in our previous papers
\cite{mp1}, \cite{mp2}.
\begin{proposition} (\cite{cp}) \label{LRtwpr}
Let $A$ and $B$ be two associative algebras and 
$R:B\otimes A\rightarrow A\otimes B$, 
$Q:A\otimes B\rightarrow A\otimes B$ two linear maps, for which 
we use a Sweedler-type notation $R(b\ot a)=a_R\ot b_R=a_r\ot b_r$
and $Q(a\otimes b)=a_Q\otimes b_Q=a_q\ot b_q$, for all $a\in A$, $b\in B$, 
satisfying the following conditions, for all $a, a'\in A$ and $b, b'\in B$:
\begin{eqnarray}
&&(aa')_R\otimes b_R=a_Ra'_r\otimes (b_R)_r, \label{tw4} \\
&&a_R\otimes (bb')_R=(a_R)_r\otimes b_rb'_R, \label{tw5} \\
&&(aa')_Q\otimes b_Q=a_qa'_Q\otimes (b_Q)_q, \label{tw4'} \\
&&a_Q\otimes (bb')_Q=(a_Q)_q\otimes b_Qb'_q, \label{tw5'} \\
&&b_R\otimes (a_R)_Q\otimes b'_Q=b_R\otimes (a_Q)_R\otimes b'_Q, 
\label{comb1} \\
&&a_R\otimes (b_R)_Q\otimes a'_Q=a_R\otimes (b_Q)_R\otimes a'_Q.
\label{comb2}
\end{eqnarray}
If we define on $A\otimes B$ a multiplication by 
$(a\otimes b)(a'\otimes b')=a_Qa'_R\otimes b_Rb'_Q$,  
then this multiplication is associative. 
This algebra structure will be denoted by 
$A\; _Q\otimes _RB$ and will be called the 
{\em  L-R-twisted tensor product} of 
$A$ and $B$ afforded by the maps $R$ and $Q$. In the particular case $Q=id_{A\ot B}$, the 
L-R-twisted tensor product $A\; _Q\otimes _RB$ reduces to the twisted tensor product $A\ot _RB$ 
introduced in \cite{Cap}, \cite{VanDaele}, whose multiplication is defined by 
$(a\ot b)(a'\ot b')=aa'_R\ot b_Rb'$.
\end{proposition}
\begin{example}
Let $H$ be a bialgebra and $D$ an $H$-bimodule algebra in the usual sense, with actions 
$H\ot D\rightarrow D$, $h\ot d\mapsto h\cdot d$ and $D\ot H\rightarrow D$, $d\ot h\mapsto d\cdot h$. 
Define the linear maps 
\begin{eqnarray*}
&&R:H\ot D\rightarrow D\ot H, \;\;\;R(h\ot d)=h_1\cdot d\ot h_2, \\
&&Q:D\ot H\rightarrow D\ot H, \;\;\;Q(d\ot h)=d\cdot h_2\ot h_1.
\end{eqnarray*}
Then we have an L-R-twisted tensor product $D\;_Q\ot _RH$, which is denoted by $D\nat H$ and is called 
the {\em L-R-smash product} of $D$ and $H$ (cf. \cite{panvan}). If we denote $d\ot h:=d\nat h$, 
for $d\in D$, $h\in H$, the multiplication of $D\nat H$ is given by 
\begin{eqnarray*}
&&(d\nat h)(d'\nat h')=(d\cdot h'_2)(h_1\cdot d')\nat h_2h'_1.
\end{eqnarray*}
If $H$ is moreover a Hopf algebra with bijective antipode, we can define as well the so-called 
{\em diagonal crossed product} $D\bowtie H$ (cf. \cite{hn1}, \cite{bpvo}), an associative algebra 
structure on $D\ot H$ 
whose multiplication is defined (we denote $d\ot h:=d\bowtie h$) by 
\begin{eqnarray*}
&&(d\bowtie h)(d'\bowtie h')=d(h_1\cdot d'\cdot S^{-1}(h_3))\bowtie h_2h'.
\end{eqnarray*}
If the action of $H$ on $D$ is unital, we have $D\nat H\simeq D\bowtie H$, cf. \cite{panvan}. 
\end{example}
\begin{definition}
(i) A {\em Hom-associative algebra} is a triple $(A, \mu , \alpha )$, in which $A$ is a linear space, 
$\alpha :A\rightarrow A$  and $\mu :A\ot A\rightarrow A$ are linear maps,  
with notation $\mu (a\ot a')=aa'$, such that 
\begin{eqnarray*}
&&\alpha (aa')=\alpha (a)\alpha (a'), \;\;\;\;\;(multiplicativity)\\
&&\alpha (a)(a'a'')=(aa')\alpha (a''), \;\;\;\;\;(Hom-associativity)
\end{eqnarray*}
for all $a, a', a''\in A$. 
We call $\alpha $ the {\em structure map} of $A$. 

A morphism $f:(A, \mu _A , \alpha _A)\rightarrow (B, \mu _B , \alpha _B)$ of Hom-associative algebras 
is a linear map $f:A\rightarrow B$ such that $\alpha _B\circ f=f\circ \alpha _A$ and 
$f\circ \mu_A=\mu _B\circ (f\ot f)$. \\
(ii) A {\em Hom-coassociative coalgebra} is a triple $(C, \Delta, \alpha )$, in which $C$ is a linear 
space, $\alpha :C\rightarrow C$ and $\Delta :C\rightarrow C\ot C$ are linear maps 
($\alpha $ is called the {\em structure map} of $C$) such that  
\begin{eqnarray*}
&&(\alpha \ot \alpha )\circ \Delta =
\Delta \circ \alpha , \;\;\;\;\;(comultiplicativity)\\  
&&(\Delta \ot \alpha )\circ \Delta =
(\alpha \ot \Delta )\circ \Delta . \;\;\;\;\;(Hom-coassociativity)
\end{eqnarray*}

A morphism $g:(C, \Delta _C , \alpha _C)\rightarrow (D, \Delta _D , \alpha _D)$ of Hom-coassociative 
coalgebras  
is a linear map $g:C\rightarrow D$ such that $\alpha _D\circ g=g\circ \alpha _C$ and 
$(g\ot g)\circ \Delta _C=\Delta _D\circ g$.
\end{definition}
\begin{remark}
Assume that $(A, \mu _A , \alpha _A)$ and $(B, \mu _B, \alpha _B)$ are two Hom-associative algebras; then 
$(A\ot B, \mu _{A\ot B}, \alpha _A\ot \alpha _B)$ is a Hom-associative algebra (called the tensor 
product of $A$ and $B$), where $\mu _{A\ot B}$ is the usual multiplication: 
$(a\ot b)(a'\ot b')=aa'\ot bb'$. 
\end{remark}
\begin{definition} 
Let $(A, \mu _A , \alpha _A)$ be a Hom-associative algebra, $M$ a linear space and $\alpha _M:M
\rightarrow M$ a linear map. \\
(i) (\cite{yau1}, \cite{homquantum3})  A {\em left $A$-module} structure on $(M, \alpha _M)$ consists of a 
linear map 
$A\ot M\rightarrow M$, $a\ot m\mapsto a\cdot m$, satisfying the conditions (for all $a, a'\in A$ and $m\in M$)
\begin{eqnarray}
&&\alpha _M(a\cdot m)=\alpha _A(a)\cdot \alpha _M(m), \label{hommod1}\\
&&\alpha _A(a)\cdot (a'\cdot m)=(aa')\cdot \alpha _M(m). \label{hommod2}
\end{eqnarray} 
(ii) (\cite{mp2}) A {\em right $A$-module} structure on $(M, \alpha _M)$ consists of a linear map 
$M\ot A\rightarrow M$, $m\ot a\mapsto m\cdot a$, satisfying the conditions (for all $a, a'\in A$ and $m\in M$) 
\begin{eqnarray}
&&\alpha _M(m\cdot a)=\alpha _M(m)\cdot \alpha _A(a), \label{righthommod1}\\
&&(m\cdot a)\cdot \alpha _A(a')=\alpha _M(m)\cdot (aa'). \label{righthommod2}
\end{eqnarray} 
If $(M, \alpha _M)$, $(N, \alpha _N)$ are left (respectively right)  $A$-modules 
($A$-actions denoted by $\cdot$),  
a morphism of left (respectively right) $A$-modules $f:M\rightarrow N$ is a linear map with
$\alpha _N\circ f=f\circ \alpha _M$  and $f(a\cdot m)=a\cdot f(m)$ (respectively  $f(m\cdot a)=f(m)\cdot a$), 
$\forall \;a\in A, \;m\in M$. 
\end{definition}

\begin{definition} (\cite{ms3}, \cite{ms4})
A {\em Hom-bialgebra} is a quadruple $(H, \mu , \Delta, \alpha )$, in which $(H, \mu , \alpha )$ is 
a Hom-associative algebra, $(H, \Delta , \alpha )$ is a Hom-coassociative coalgebra  
and moreover $\Delta $ is a morphism of Hom-associative algebras.  
\end{definition}

Thus, a Hom-bialgebra is a Hom-associative algebra $(H, \mu , \alpha )$ endowed with a 
linear map $\Delta :H\rightarrow H\ot H$, with notation $\Delta (h)=h_1\ot h_2$, such that, for all $h, h'\in H$, 
we have: 
\begin{eqnarray}
&&\Delta (h_1)\ot \alpha (h_2)=\alpha (h_1)\ot \Delta (h_2),  \label{hombia1}\\
&&\Delta (hh')=h_1h'_1\ot h_2h'_2, \label{hombia2}\\
&&\Delta (\alpha (h))=\alpha (h_1)\ot \alpha (h_2). \label{hombia3}
\end{eqnarray}
\begin{proposition} (\cite{ms4}, \cite{yau3}) \label{yautwisting}
(i) Let $(A, \mu )$ be an associative algebra and $\alpha :A\rightarrow A$ an algebra endomorphism. Define 
a new multiplication $\mu _{\alpha }:=\alpha \circ \mu :A\ot A\rightarrow A$. Then 
$(A, \mu _{\alpha }, \alpha )$ is a Hom-associative algebra, denoted by $A_{\alpha }$. \\
(ii) Let $(C, \Delta )$ be a coassociative coalgebra and $\alpha :C\rightarrow C$ a 
coalgebra endomorphism. Define 
a new comultiplication $\Delta _{\alpha }:=\Delta \circ \alpha :C\rightarrow C\ot C$. Then 
$(C, \Delta _{\alpha }, \alpha )$ is a Hom-coassociative coalgebra, denoted by $C_{\alpha }$. \\
(iii) Let $(H, \mu , \Delta )$ be a bialgebra and $\alpha :H\rightarrow H$ a bialgebra endomorphism. 
If we define $\mu _{\alpha }$ and $\Delta _{\alpha }$ as in (i) and (ii), then 
$H_{\alpha }=(H, \mu _{\alpha }, \Delta _{\alpha }, \alpha )$ is a Hom-bialgebra. 
\end{proposition}
\begin{proposition} (\cite{homquantum3}) \label{tensprodmod}
Let $(H, \mu _H, \Delta _H, \alpha _H)$ be a Hom-bialgebra. If $(M, \alpha _M)$ and 
$(N, \alpha _N)$ are left $H$-modules, then $(M\ot N, \alpha _M\ot \alpha _N)$ is also a 
left $H$-module, with $H$-action defined by $H\ot (M\ot N)\rightarrow M\ot N$, $h\ot (m\ot n)\mapsto 
h\cdot (m\ot n):=h_1\cdot m\ot h_2\cdot n$. 
\end{proposition}
\begin{definition} (\cite{yau1}) 
Let $(H, \mu _H, \Delta _H, \alpha _H)$ be a Hom-bialgebra. A Hom-associative algebra 
$(A, \mu _A, \alpha _A)$ is called a {\em left $H$-module Hom-algebra} if $(A, \alpha _A)$ is a 
left $H$-module, 
with action denoted by $H\ot A\rightarrow A$, $h\ot a\mapsto h\cdot a$, such that the following 
condition is satisfied: 
\begin{eqnarray}
&&\alpha _H^2(h)\cdot (aa')=(h_1\cdot a)(h_2\cdot a'), \;\;\;\forall \;h\in H, \;a, a'\in A. \label{modalgcompat}
\end{eqnarray} 
\end{definition}
\begin{proposition} (\cite{yau1}) \label{deformmodalg}
Let $(H, \mu _H, \Delta _H)$ be a bialgebra and $(A, \mu _A)$ a left $H$-module algebra in the usual sense, 
with action denoted by $H\ot A\rightarrow A$, $h\ot a\mapsto h\cdot a$. Let $\alpha _H:H\rightarrow H$ 
be a bialgebra endomorphism and $\alpha _A:A\rightarrow A$ an algebra endomorphism, such that 
$\alpha _A(h\cdot a)=\alpha _H(h)\cdot \alpha _A(a)$, for all $h\in H$ and $a\in A$. If we 
consider the Hom-bialgebra 
$H_{\alpha _H}=(H, \alpha _H\circ \mu _H, \Delta _H\circ \alpha _H, \alpha _H)$ and 
the Hom-associative algebra $A_{\alpha _A}=(A, \alpha _A\circ \mu _A, \alpha _A)$, 
then $A_{\alpha _A}$ is a left $H_{\alpha _H}$-module Hom-algebra in the above sense, with action 
$H_{\alpha _H}\ot A_{\alpha _A}\rightarrow A_{\alpha _A}$, $h\ot a\mapsto h\triangleright a:=
\alpha _A(h\cdot a)=\alpha _H(h)\cdot \alpha _A(a)$. 
\end{proposition}
\begin{definition} (\cite{mp2})
Assume that $(H, \mu _H, \Delta _H, \alpha _H)$ is a Hom-bialgebra. A Hom-associative algebra 
$(C, \mu _C, \alpha _C)$ is called a {\em right $H$-module Hom-algebra} if $(C, \alpha _C)$ is a 
right $H$-module, 
with action denoted by $C\ot H\rightarrow C$, $c\ot h\mapsto c\cdot h$, such that the following 
condition is satisfied: 
\begin{eqnarray}
&&(cc')\cdot \alpha _H^2(h)=(c\cdot h_1)(c'\cdot h_2), \;\;\;\forall \;h\in H, \;c, c'\in C. 
\end{eqnarray} 
\end{definition}
\begin{proposition}  (\cite{mp2}) \label{rightdefmodalg}
Let $(H, \mu _H, \Delta _H)$ be a bialgebra and $(C, \mu _C)$ a right $H$-module algebra in the usual sense, 
with action denoted by $C\ot H\rightarrow C$, $c\ot h\mapsto c\cdot h$. Let $\alpha _H:H\rightarrow H$ 
be a bialgebra endomorphism and $\alpha _C:C\rightarrow C$ an algebra endomorphism, such that 
$\alpha _C(c\cdot h)=\alpha _C(c)\cdot \alpha _H(h)$, for all $h\in H$ and $c\in C$. Then 
the Hom-associative algebra $C_{\alpha _C}=(C, \alpha _C\circ \mu _C, \alpha _C)$  
becomes a right module Hom-algebra over the Hom-bialgebra 
$H_{\alpha _H}=(H, \alpha _H\circ \mu _H, \Delta _H\circ \alpha _H, \alpha _H)$, 
with action defined by  
$C_{\alpha _C}\ot H_{\alpha _H}\rightarrow C_{\alpha _C}$, $c\ot h\mapsto c\triangleleft h:=
\alpha _C(c\cdot h)=\alpha _C(c)\cdot \alpha _H(h)$. 
\end{proposition}
\begin{proposition} (\cite{mp2}) \label{Def-HomTwistor}
Let $(D, \mu , \alpha )$ be a Hom-associative algebra and $T:D\otimes
D\rightarrow D\otimes D$ a linear map, with notation $T(d\ot d')=d^T\ot d'_T$, for 
$d, d'\in D$, 
satisfying the conditions
\begin{eqnarray}
&&(\alpha \ot \alpha )\circ T=T\circ (\alpha \ot \alpha ), \label{multtwistor} \\
&&T\circ (\alpha \otimes \mu )=
(\alpha \otimes \mu )\circ T_{13}\circ T_{12},
\label{homtwistor1} \\
&&T\circ (\mu \otimes \alpha )=
(\mu \otimes \alpha )\circ T_{13}\circ T_{23},
\label{homtwistor2} \\
&&T_{12}\circ T_{23}=T_{23}\circ T_{12},  \label{homtwistor3}
\end{eqnarray}
where we used a standard notation for the operators $T_{ij}$, namely $T_{12}=T\ot id_D$, 
$T_{23}=id_D\ot T$ and $T_{13}(d\ot d'\ot d'')=d^T\ot d'\ot d''_T$.
Then $D^T:=(D, \mu \circ T, \alpha )$ is also a Hom-associative algebra. 
The map $T$ is called a {\em Hom-twistor}.
\end{proposition}
\begin{proposition} (\cite{mp2})
Let $(A, \mu _A, \alpha _A)$ and $(B, \mu _B, \alpha _B)$ 
be two Hom-associative algebras and $R:B\ot A 
\rightarrow A\ot B$ a linear map, with Sweedler-type notation $R(b\ot a)=a_R\ot b_R=a_r\ot b_r$, for 
$a\in A$, $b\in B$. Assume that the following conditions are satisfied: 
\begin{eqnarray}
&&\alpha _A(a_R)\ot \alpha _B(b_R)=\alpha _A(a)_R\ot \alpha _B(b)_R, \label{homsweed0} \\
&&(aa')_R\ot \alpha _B(b)_R=a_Ra'_r\ot \alpha _B((b_R)_r), \label{homsweed1} \\
&&\alpha _A(a)_R\ot (bb')_R=\alpha _A((a_R)_r)\ot b_rb'_R, \label{homsweed2}
\end{eqnarray}
for all $a, a'\in A$ and $b, b'\in B$ (such a map $R$ is called a {\em Hom-twisting map}). If we define a new 
multiplication on $A\ot B$ by 
$(a\ot b)(a'\ot b')=aa'_R\ot b_Rb'$, then $A\ot B$ becomes a Hom-associative algebra with 
structure map $\alpha _A\ot \alpha _B$, denoted by $A\ot _RB$ and called the {\em Hom-twisted tensor 
product} of $A$ and $B$. 
\end{proposition}
\begin{theorem} (\cite{mp2}) \label{homiterated}
Let $(A, \mu _A, \alpha _A)$, $(B, \mu _B, \alpha _B)$ and $(C, \mu _C, \alpha _C)$ be three  
Hom-associative algebras and $R_1:B\ot A\rightarrow A\ot B$, $R_2:C\ot B\rightarrow B\ot C$, 
$R_3:C\ot A\rightarrow A\ot C$ three Hom-twisting maps, satisfying the braid condition 
\begin{eqnarray}
&&(id_A\ot R_2)\circ (R_3\ot id_B)\circ (id_C\ot R_1)=
(R_1\ot id_C)\circ (id_B\ot R_3)\circ (R_2\ot id_A). \label{hombraid}
\end{eqnarray}
Define the maps
\begin{eqnarray*}
&&P_1:C\ot (A\ot _{R_1}B)\rightarrow (A\ot _{R_1}B)\ot C, \;\;\;P_1=(id_A\ot R_2)\circ (R_3\ot id_B), \\
&&P_2:(B\ot _{R_2}C)\ot A\rightarrow A\ot (B\ot _{R_2}C), \;\;\;P_2=(R_1\ot id_C)\circ (id_B\ot R_3).
\end{eqnarray*} 
Then $P_1$ is a Hom-twisting map between $A\ot _{R_1}B$ and $C$, $P_2$ is a 
Hom-twisting map between $A$ and $B\ot _{R_2}C$, and the Hom-associative algebras 
$(A\ot _{R_1}B)\ot _{P_1}C$ and $A\ot _{P_2}(B\ot _{R_2}C)$ coincide; this Hom-associative 
algebra will be denoted by $A\ot _{R_1}B\ot _{R_2}C$ and will be called the {\em iterated Hom-twisted tensor 
product} of $A$, $B$, $C$. 
\end{theorem}
\begin{proposition} (\cite{mp2})
Let $(H, \mu _H, \Delta _H, \alpha _H)$ be a Hom-bialgebra, $(A, \mu _A, \alpha _A)$ a left 
$H$-module Hom-algebra and 
$(C, \mu _C, \alpha _C)$ a right $H$-module Hom-algebra, with actions denoted by 
$H\ot A\rightarrow A$, $h\ot a\mapsto h\cdot a$ and 
$C\ot H\rightarrow C$, $c\ot h\mapsto c\cdot h$, and assume that the structure maps 
$\alpha _H$, $\alpha _A$, $\alpha _C$ are bijective. 
Then:\\
(i) We have the following Hom-twisting maps:
\begin{eqnarray*}
&&R_1:H\ot A\rightarrow A\ot H, \;\;\;R_1(h\ot a)=\alpha _H^{-2}(h_1)\cdot \alpha _A^{-1}(a)\ot 
\alpha _H^{-1}(h_2), \\
&&R_2:C\ot H\rightarrow H\ot C, \;\;\;R_2(c\ot h)=\alpha _H^{-1}(h_1) \ot \alpha _C^{-1}(c)\cdot 
\alpha _H^{-2}(h_2). 
\end{eqnarray*}
Thus, we can consider the Hom-associative algebras $A\ot _{R_1}H$ and 
$H\ot _{R_2}C$, which are denoted by $A\# H$ and respectively $H\# C$ and are called 
the left and respectively right {\em Hom-smash products}. If we denote $a\ot h:=a\# h$ and 
$h\# c:=h\ot c$, for $a\in A$, $h\in H$, $c\in C$, the multiplications of the smash products are given by 
\begin{eqnarray*}
&&(a\# h)(a'\# h')=a(\alpha _H^{-2}(h_1)\cdot \alpha _A^{-1}(a'))\# \alpha _H^{-1}(h_2)h', \\
&&(h\# c)(h'\# c')=h\alpha _H^{-1}(h'_1)\# (\alpha _C^{-1}(c)\cdot \alpha _H^{-2}(h'_2))c', 
\end{eqnarray*}
and the structure maps are $\alpha _A\ot \alpha _H$ and respectively $\alpha _H\ot \alpha _C$. \\
(ii) Consider as well the trivial Hom-twisting map $R_3:C\ot A\rightarrow A\ot C$, $R_3(c\ot a)=a\ot c$. Then 
$R_1$, $R_2$, $R_3$ satisfy the braid relation, so we can consider 
the iterated Hom-twisted tensor product $A\ot _{R_1}H\ot _{R_2}C$, which is denoted by 
$A\# H\# C$ and is called the {\em two-sided Hom-smash product}. Its structure map is 
$\alpha _A\ot \alpha _H\ot \alpha _C$ and 
its multiplication is defined by 
\begin{eqnarray*}
&&(a\# h\# c)(a'\# h'\# c')=a(\alpha _H^{-2}(h_1)\cdot \alpha _A^{-1}(a'))\# 
\alpha _H^{-1}(h_2h'_1)\#  (\alpha _C^{-1}(c)\cdot \alpha _H^{-2}(h'_2))c'.
\end{eqnarray*} 
\end{proposition}
\begin{definition} (\cite{homquantum1}, \cite{homquantum2}) 
Let $(H, \mu , \Delta, \alpha )$ be a Hom-bialgebra and $R\in H\ot H$ an element, with 
Sweedler-type notation $R=R^1\ot R^2=r^1\ot r^2$. Then $(H, \mu , \Delta, \alpha , R)$ is called 
{\em quasitriangular Hom-bialgebra} if the following axioms are satisfied: 
\begin{eqnarray}
&&(\Delta \ot \alpha )(R)=\alpha (R^1)\ot \alpha (r^1)\ot R^2r^2, \label{homQT1} \\
&&(\alpha \ot \Delta )(R)=R^1r^1\ot \alpha (r^2)\ot \alpha (R^2), \label{homQT2} \\
&&\Delta ^{cop}(h)R=R\Delta (h), \label{homQT3}
\end{eqnarray}
for all $h\in H$, where we denoted as usual $\Delta ^{cop }(h)=h_2\ot h_1$. 
\end{definition}
\section{Hom-L-R-twisted tensor products of algebras}
\setcounter{equation}{0}
${\;\;\;}$We introduce the Hom-analogue of Proposition \ref{LRtwpr}.
\begin{proposition}\label{homlrttp}
Let $(A, \mu _A, \alpha _A)$ and $(B, \mu _B, \alpha _B)$ 
be two Hom-associative algebras and $R:B\ot A 
\rightarrow A\ot B$, $Q:A\ot B\rightarrow A\ot B$  two linear maps, with 
notation $R(b\ot a)=a_R\ot b_R=a_r\ot b_r$ and 
$Q(a\otimes b)=a_Q\otimes b_Q=a_q\ot b_q$, for all $a\in A$, $b\in B$, satisfying the conditions: 
\begin{eqnarray}
&&\alpha _A(a_R)\ot \alpha _B(b_R)=\alpha _A(a)_R\ot \alpha _B(b)_R, \label{lrhom1} \\
&&\alpha _A(a_Q)\ot \alpha _B(b_Q)=\alpha _A(a)_Q\ot \alpha _B(b)_Q, \label{lrhom2} \\
&&(aa')_R\ot \alpha _B(b)_R=a_Ra'_r\ot \alpha _B((b_R)_r), \label{lrhom3} \\
&&\alpha _A(a)_R\ot (bb')_R=\alpha _A((a_R)_r)\ot b_rb'_R, \label{lrhom4}\\
&&(aa')_Q\ot \alpha _B(b)_Q=a_qa'_Q\ot \alpha _B((b_Q)_q), \label{lrhom5} \\
&&\alpha _A(a)_Q\ot (bb')_Q=\alpha _A((a_Q)_q)\ot b_Qb'_q, \label{lrhom6}\\
&&b_R\otimes (a_R)_Q\otimes b'_Q=b_R\otimes (a_Q)_R\otimes b'_Q, 
\label{lrhom7} \\
&&a_R\otimes (b_R)_Q\otimes a'_Q=a_R\otimes (b_Q)_R\otimes a'_Q, 
\label{lrhom8}
\end{eqnarray}
for all $a, a'\in A$ and $b, b'\in B$. Define a new 
multiplication on $A\ot B$ by 
$(a\ot b)(a'\ot b')=a_Qa'_R\ot b_Rb'_Q$. Then $A\ot B$ with this multiplication 
is a Hom-associative algebra with 
structure map $\alpha _A\ot \alpha _B$, denoted by $A\; _Q\otimes _RB$ and called the 
{\em Hom-L-R-twisted tensor 
product} of $A$ and $B$ afforded by the maps $R$ and $Q$. 
\end{proposition}
\begin{proof}
The fact that $\alpha _A\ot \alpha _B$ is multiplicative follows immediately from (\ref{lrhom1}) and 
(\ref{lrhom2}). Now we compute (denoting $R=r=\mathcal{R}$ and $Q=q=\tilde{Q}$):\\[2mm]
${\;\;\;}$$(\alpha _A\ot \alpha _B)(a\ot b)[(a'\ot b')(a''\ot b'')]$
\begin{eqnarray*}
&=&(\alpha _A(a)\ot \alpha _B(b))(a'_Qa''_R\ot b'_Rb''_Q)\\
&=&\alpha _A(a)_q(a'_Qa''_R)_r\ot \alpha _B(b)_r(b'_Rb''_Q)_q\\
&\overset{(\ref{lrhom3}), \;(\ref{lrhom6})}{=}&\alpha _A((a_{\tilde{Q}})_q)((a'_Q)_{\mathcal{R}}
(a''_R)_r)\ot \alpha _B((b_{\mathcal{R}})_r)((b'_R)_{\tilde{Q}}(b''_Q)_q)\\
&\overset{Hom-assoc.}{=}&((a_{\tilde{Q}})_q(a'_Q)_{\mathcal{R}})
\alpha _A((a''_R)_r)\ot ((b_{\mathcal{R}})_r(b'_R)_{\tilde{Q}})\alpha _B((b''_Q)_q)\\
&\overset{(\ref{lrhom7}), \;(\ref{lrhom8})}{=}&((a_{\tilde{Q}})_q(a'_{\mathcal{R}})_Q)
\alpha _A((a''_R)_r)\ot ((b_{\mathcal{R}})_r(b'_{\tilde{Q}})_R)\alpha _B((b''_Q)_q)\\
&\overset{(\ref{lrhom4}), \;(\ref{lrhom5})}{=}&(a_{\tilde{Q}}a'_{\mathcal{R}})_Q
\alpha _A(a'')_R\ot (b_{\mathcal{R}}b'_{\tilde{Q}})_R\alpha _B(b'')_Q\\
&=&(a_{\tilde{Q}}a'_{\mathcal{R}}\ot b_{\mathcal{R}}b'_{\tilde{Q}})(\alpha _A(a'')\ot \alpha _B(b''))\\
&=&[(a\ot b)(a'\ot b')](\alpha _A\ot \alpha _B)(a''\ot b''),
\end{eqnarray*}
finishing the proof.
\end{proof}

Obviously, a Hom-twisted tensor product $A\ot _RB$ is a particular case of a Hom-L-R-twisted tensor product, 
namely $A\ot _RB=A\; _Q\otimes _RB$, where $Q=id_{A\ot B}$. On the other hand, if $A$ and $B$ 
are Hom-associative algebras and $Q:A\ot B\rightarrow A\ot B$ is a linear map satisfying the conditions 
(\ref{lrhom2}), (\ref{lrhom5}) and (\ref{lrhom6}), then the multiplication $(a\ot b)(a'\ot b')=a_Qa'\ot bb'_Q$ 
defines a Hom-associative algebra structure on $A\ot B$, denoted by $A\; _Q\otimes B$; this is a 
particular case of a  Hom-L-R-twisted tensor product, namely $A\; _Q\otimes B=A\; _Q\otimes _RB$, 
where $R$ is the flip map $b\ot a\mapsto a\ot b$. Also, if $A\; _Q\otimes _RB$ is a 
Hom-L-R-twisted tensor product, we can consider as well the Hom-associative algebras $A\ot _RB$ and 
$A\; _Q\otimes B$. 

By using some computations similar to the ones performed in \cite{mp2}, 
Propositions 2.6 and 2.10, one can prove the following two results:
\begin{proposition}
Let $(A, \mu _A, \alpha _A)$ and $(B, \mu _B, \alpha _B)$ 
be Hom-associative algebras and $A\; _Q\otimes _RB$ a Hom-L-R-twisted tensor product. Define the linear maps 
$T, U, V:(A\ot B)\ot (A\ot B)\rightarrow (A\ot B)\ot (A\ot B)$, by 
\begin{eqnarray*}
&&T((a\ot b)\ot (a'\ot b'))=(a_Q\ot b_R)\ot (a'_R\ot b'_Q), \\
&&U((a\ot b)\ot (a'\ot b'))=(a_Q\ot b)\ot (a'\ot b'_Q), \\
&&V((a\ot b)\ot (a'\ot b'))=(a\ot b_R)\ot (a'_R\ot b'). 
\end{eqnarray*}
Then $T$ is a Hom-twistor for $A\ot B$, $U$ is a Hom-twistor for $A\ot _RB$, $V$ is a Hom-twistor for 
$A\; _Q\otimes B$ and $A\; _Q\otimes _RB=(A\ot B)^T=(A\ot _RB)^U=(A\; _Q\otimes B)^V$ as 
Hom-associative algebras. 
\end{proposition}
\begin{proposition}\label{deformLRttp}
Let $(A, \mu _A)$ and $(B, \mu _B)$ be two  associative algebras, $\alpha _A:A\rightarrow A$ and 
$\alpha _B:B\rightarrow B$ algebra maps and $A\; _Q\otimes _RB$ an L-R-twisted tensor product
such that $(\alpha _A\ot \alpha _B)\circ R=R\circ (\alpha _B\ot \alpha _A)$ and  
$(\alpha _A\ot \alpha _B)\circ Q=Q\circ (\alpha _A\ot \alpha _B)$. Then we have a 
Hom-L-R-twisted tensor product $A_{\alpha _A}\;_Q\ot _RB_{\alpha _B}$, which coincides with 
$(A\;_Q\ot _RB)_{\alpha _A\ot \alpha _B}$ as Hom-associative algebras.
\end{proposition}

${\;\;\;}$The next result is the Hom-analogue of \cite{cp}, Proposition 2.9.
\begin{proposition} \label{biject}
Let $A\; _Q\otimes _RB$ be a Hom-L-R-twisted tensor product of the Hom-associative algebras 
$(A, \mu _A, \alpha _A)$ and $(B, \mu _B, \alpha _B)$ with bijective structure maps $\alpha _A$ and 
$\alpha _B$ and assume that $Q$ is bijective with inverse $Q^{-1}$. 
Then the map $P:B\otimes A\rightarrow A\otimes B$ defined by $P=Q^{-1}\circ R$ is a  
Hom-twisting map, and we have an isomorphism of Hom-associative algebras 
$Q:A\otimes _PB\simeq A\; _Q\otimes _RB$.
\end{proposition}
\begin{proof}
Let $a, a'\in A$ and $b, b'\in B$; we denote $Q^{-1}(a\ot b)=a_{Q^{-1}}\ot b_{Q^{-1}}=a_{q^{-1}}\ot 
b_{q^{-1}}$, $P(b\ot a)=a_P\ot b_P=a_p\ot b_p$, $Q=q=\overline{Q}=\overline{q}$. 
The relation (\ref{homsweed0}) for $P$ follows immediately from (\ref{lrhom1}) and (\ref{lrhom2}). 
Now we compute: 
\begin{eqnarray*}
Q(a_Pa'_p\ot \alpha _B((b_P)_p))&=&((a_R)_{Q^{-1}}(a'_r)_{q^{-1}})_Q\ot 
\alpha _B((((b_R)_{Q^{-1}})_r)_{q^{-1}})_Q\\
&\overset{(\ref{lrhom5})}{=}&((a_R)_{Q^{-1}})_Q((a'_r)_{q^{-1}})_{\overline{Q}}\ot 
\alpha _B((((((b_R)_{Q^{-1}})_r)_{q^{-1}})_{\overline{Q}})_Q)\\
&=&((a_R)_{Q^{-1}})_Qa'_r\ot \alpha _B((((b_R)_{Q^{-1}})_r)_Q)\\
&\overset{(\ref{lrhom8})}{=}&((a_R)_{Q^{-1}})_Qa'_r\ot \alpha _B((((b_R)_{Q^{-1}})_Q)_r)\\
&=&a_Ra'_r\ot \alpha _B((b_R)_r)\\
&\overset{(\ref{lrhom3})}{=}&(aa')_R\ot \alpha _B(b)_R=R(\alpha _B(b)\ot aa').
\end{eqnarray*}
By applying $Q^{-1}$ to this equality we obtain $P(\alpha _B(b)\ot aa')=a_Pa'_p\ot \alpha _B((b_P)_p)$, 
which is (\ref{homsweed1}) for $P$. Similarly one can prove (\ref{homsweed2}) for $P$, so $P$ is indeed 
a Hom-twisting map. Since $Q$ satisfies (\ref{lrhom2}), the only thing left to prove is that 
$Q:A\otimes _PB\rightarrow A\; _Q\otimes _RB$ is multiplicative. We compute:\\[2mm]
${\;\;\;}$
$Q((a\ot b)(a'\ot b'))$
\begin{eqnarray*}
&=&(aa'_P)_Q\ot (b_Pb')_Q\\
&=&(aa'_P)_Q\ot \alpha _B(\alpha _B^{-1}(b_Pb'))_Q\\
&\overset{(\ref{lrhom5})}{=}&a_q(a'_P)_Q\ot \alpha _B((\alpha _B^{-1}(b_Pb')_Q)_q)\\
&=&a_q(a'_P)_Q\ot \alpha _B(((\alpha _B^{-1}(b_P)\alpha _B^{-1}(b'))_Q)_q)\\
&=&a_q\alpha _A(\alpha _A^{-1}(a'_P))_Q\ot \alpha _B(((\alpha _B^{-1}(b_P)\alpha _B^{-1}(b'))_Q)_q)\\
&\overset{(\ref{lrhom6})}{=}&a_q\alpha _A((\alpha _A^{-1}(a'_P)_Q)_{\overline{Q}})\ot 
\alpha _B((\alpha _B^{-1}(b_P)_Q\alpha _B^{-1}(b')_{\overline{Q}})_q)\\
&=&\alpha _A(\alpha _A^{-1}(a))_q\alpha _A((\alpha _A^{-1}(a'_P)_Q)_{\overline{Q}})\ot 
\alpha _B((\alpha _B^{-1}(b_P)_Q\alpha _B^{-1}(b')_{\overline{Q}})_q)\\
&\overset{(\ref{lrhom6})}{=}&\alpha _A((\alpha _A^{-1}(a)_q)_{\overline{q}})
\alpha _A((\alpha _A^{-1}(a'_P)_Q)_{\overline{Q}})\ot 
\alpha _B((\alpha _B^{-1}(b_P)_Q)_q(\alpha _B^{-1}(b')_{\overline{Q}})_{\overline{q}})\\
&\overset{(\ref{homsweed0})}{=}&\alpha _A((\alpha _A^{-1}(a)_q)_{\overline{q}})
\alpha _A(((\alpha _A^{-1}(a')_P)_Q)_{\overline{Q}})\ot 
\alpha _B(((\alpha _B^{-1}(b)_P)_Q)_q(\alpha _B^{-1}(b')_{\overline{Q}})_{\overline{q}})\\
&=&\alpha _A((\alpha _A^{-1}(a)_q)_{\overline{q}})
\alpha _A((\alpha _A^{-1}(a')_R)_{\overline{Q}})\ot 
\alpha _B((\alpha _B^{-1}(b)_R)_q(\alpha _B^{-1}(b')_{\overline{Q}})_{\overline{q}})\\
&\overset{(\ref{lrhom7})}{=}&\alpha _A((\alpha _A^{-1}(a)_q)_{\overline{q}})
\alpha _A((\alpha _A^{-1}(a')_{\overline{Q}})_R)\ot 
\alpha _B((\alpha _B^{-1}(b)_R)_q(\alpha _B^{-1}(b')_{\overline{Q}})_{\overline{q}})\\
&\overset{(\ref{lrhom8})}{=}&\alpha _A((\alpha _A^{-1}(a)_q)_{\overline{q}}
(\alpha _A^{-1}(a')_{\overline{Q}})_R)\ot 
\alpha _B((\alpha _B^{-1}(b)_q)_R(\alpha _B^{-1}(b')_{\overline{Q}})_{\overline{q}})\\
&\overset{(\ref{lrhom2})}{=}&\alpha _A(\alpha _A^{-1}(a_q)_{\overline{q}}
\alpha _A^{-1}(a'_{\overline{Q}})_R)\ot 
\alpha _B(\alpha _B^{-1}(b_q)_R\alpha _B^{-1}(b'_{\overline{Q}})_{\overline{q}})\\
&\overset{(\ref{lrhom1}), \;(\ref{lrhom2})}{=}&\alpha _A(\alpha _A^{-1}((a_q)_{\overline{q}})
\alpha _A^{-1}((a'_{\overline{Q}})_R))\ot 
\alpha _B(\alpha _B^{-1}((b_q)_R)\alpha _B^{-1}((b'_{\overline{Q}})_{\overline{q}}))\\
&=&(a_q)_{\overline{q}}(a'_{\overline{Q}})_R\ot 
(b_q)_R(b'_{\overline{Q}})_{\overline{q}}\\
&=&(a_q\ot b_q)(a'_{\overline{Q}}\ot b'_{\overline{Q}})
=Q(a\ot b)Q(a'\ot b'),
\end{eqnarray*}
finishing the proof.
\end{proof}
\begin{proposition}\label{mis}
Let $A\ot _{R_1}B\ot _{R_2}C$ be an iterated Hom-twisted tensor product of the Hom-associative algebras 
$(A, \mu _A, \alpha _A)$, $(B, \mu _B, \alpha _B)$, $(C, \mu _C, \alpha _C)$ for which the map 
$R_3:C\ot A\rightarrow A\ot C$ is the flip map $c\ot a\mapsto a\ot c$. Define the linear maps 
\begin{eqnarray*}
&&R:B\ot (A\ot C)\rightarrow (A\ot C)\ot B, \;\;\;R(b\ot (a\ot c))=(a_{R_1}\ot c)\ot b_{R_1}, \\
&&Q:(A\ot C)\ot B\rightarrow (A\ot C)\ot B, \;\;\;Q((a\ot c)\ot b)=(a\ot c_{R_2})\ot b_{R_2}.
\end{eqnarray*}
Then we have a Hom-L-R-twisted tensor product $(A\ot C)\;_Q\ot _RB$, and an isomorphism of 
Hom-associative algebras $A\ot _{R_1}B\ot _{R_2}C\simeq (A\ot C)\;_Q\ot _RB$, 
$a\ot b\ot c \mapsto (a\ot c)\ot b$.
\end{proposition}
\begin{proof}
One has to prove the relations (\ref{lrhom1})-(\ref{lrhom8}) for $R$ and $Q$. The relations 
(\ref{lrhom1})-(\ref{lrhom6}) are easy consequences of the fact that $R_1$ and $R_2$ are 
Hom-twisting maps, (\ref{lrhom7}) is trivially satisfied while (\ref{lrhom8}) is a consequence of the 
braid relation, which in our situation ($R_3$ is the flip) boils down to
$a_{R_1}\ot (b_{R_1})_{R_2}\ot c_{R_2}=a_{R_1}\ot (b_{R_2})_{R_1}\ot c_{R_2}$, 
for all $a\in A$, $b\in B$, $c\in C$. So indeed $(A\ot C)\;_Q\ot _RB$ is a Hom-L-R-twisted tensor product, 
with multiplication 
\begin{eqnarray*}
&&((a\ot c)\ot b)((a'\ot c')\ot b')=(aa'_{R_1}\ot c_{R_2}c')\ot b_{R_1}b'_{R_2}. 
\end{eqnarray*}
Again because $R_3$ is the flip, the multiplication in the iterated Hom-twisted tensor product 
$A\ot _{R_1}B\ot _{R_2}C$ is given by 
$(a\ot b\ot c)(a'\ot b'\ot c')=aa'_{R_1}\ot b_{R_1}b'_{R_2}\ot c_{R_2}c'$, 
so obviously we have $A\ot _{R_1}B\ot _{R_2}C\simeq (A\ot C)\;_Q\ot _RB$.
\end{proof}
\section{Hom-L-R-smash product}
\setcounter{equation}{0}
\begin{definition}\label{hombimodul}
Let $(A, \mu _A, \alpha _A)$ be a Hom-associative algebra, $M$ a linear space and 
$\alpha _M:M\rightarrow M$ a linear map. Assume that $(M, \alpha _M)$ is both a left and a right 
$A$-module (with actions denoted by $A\ot M\rightarrow M$, $a\ot m\mapsto a\cdot m$ and $M\ot A
\rightarrow M$, $m\ot a\mapsto m\cdot a$). We call $(M, \alpha _M)$ an {\em $A$-bimodule} if 
the following condition is satisfied, for all $a, a'\in A$, $m\in M$:
\begin{eqnarray}
&&\alpha _A(a)\cdot (m\cdot a')=(a\cdot m)\cdot \alpha _A(a').
\label{hombimodule}
\end{eqnarray} 
\end{definition}
\begin{remark}
Obviously, $(A, \alpha _A)$ is an $A$-bimodule. 
\end{remark}

We prove now that this is indeed the ''appropriate'' concept of bimodule for the class of Hom-associative 
algebras. Recall first from \cite{schafer} the following concept. Let $\cal C$ be a  
class of (not necessarily associative) algebras, $A\in \cal C$ and $M$ a 
linear space with two linear actions $a\otimes m\mapsto a\cdot m$ 
and $m\otimes a \mapsto m\cdot a$ of $A$ on $M$. On the 
direct sum $A\oplus M$ one can introduce an algebra structure (called the 
{\it semidirect sum} or {\it split null extension}) by defining a 
multiplication in $A\oplus M$ by 
\begin{eqnarray*}
(a,m)(a',m')=(aa', m\cdot a'+a\cdot m'), 
\end{eqnarray*}
for all $a, a'\in A$ and $m, m'\in M$. Then, if $A\oplus M$ with this 
algebra structure is in $\cal C$, we say that $M$ is an $A$-bimodule with 
respect to $\cal C$. If $\cal C$ is the class of all  
associative algebras or of all Lie algebras, one obtains the usual concepts 
of bimodule for these types of algebras. We have then the following result:
\begin{proposition}
Let $(A, \mu _A, \alpha _A)$ be a Hom-associative algebra, $M$ a linear space, 
$\alpha _M:M\rightarrow M$ a linear map and two linear actions $a\otimes m\mapsto a\cdot m$ 
and $m\otimes a \mapsto m\cdot a$ of $A$ on $M$. Then the split null extension $B=A\oplus M$ 
is a Hom-associative algebra with structure map $\alpha _B$ defined by $\alpha _B((a, m))=
(\alpha _A(a), \alpha _M(m))$ if and only if $(M, \alpha _M)$ is an $A$-bimodule as in Definition 
\ref{hombimodul}.
\end{proposition}
\begin{proof}
It is easy to see that 
\begin{eqnarray*}
&&\alpha _B((a, m))[(a', m')(a'', m'')]=(\alpha _A(a)(a'a''), \alpha _M(m)\cdot (a'a'')+
\alpha _A(a)\cdot (m'\cdot a'')\\
&&\;\;\;\;\;\;\;\;\;\;\;\;\;\;\;\;\;\;\;\;\;\;\;\;\;\;\;\;\;\;\;\;\;\;\;\;\;\;\;\;\;\;\;\;\;\;\;\;\;\;\;\;\;
+\alpha _A(a)\cdot (a'\cdot m'')), \\
&&[(a, m)(a', m')]\alpha _B((a'', m''))=((aa')\alpha _A(a''), (m\cdot a')\cdot \alpha _A(a'')+
(a\cdot m')\cdot \alpha _A(a'')\\
&&\;\;\;\;\;\;\;\;\;\;\;\;\;\;\;\;\;\;\;\;\;\;\;\;\;\;\;\;\;\;\;\;\;\;\;\;\;\;\;\;\;\;\;\;\;\;\;\;\;\;\;\;\;
+(aa')\cdot \alpha _M(m'')),
\end{eqnarray*}
so the multiplication on $B$ is Hom-associative if and only if 
$$\alpha _M(m)\cdot (a'a'')+\alpha _A(a)\cdot (m'\cdot a'')+\alpha _A(a)\cdot (a'\cdot m'')=
(m\cdot a')\cdot \alpha _A(a'')+(a\cdot m')\cdot \alpha _A(a'')+(aa')\cdot \alpha _M(m'')$$
for all $a, a', a''\in A$ and $m, m', m''\in M$. Also, clearly $\alpha _B$ is multiplicative 
if and only if 
$$\alpha _M(m\cdot a')+\alpha _M(a\cdot m')=\alpha _M(m)\cdot \alpha _A(a')+
\alpha _A(a)\cdot \alpha _M(m')$$
for all $a, a'\in A$ and $m, m'\in M$. It is then obvious that if $(M, \alpha _M)$ is an $A$-bimodule then 
$B$ is Hom-associative. Conversely, assuming the two relations above, we take $m'=m''=0$, then 
$m=m'=0$, then $m=m''=0$ in the first relation and $m=0$, then $m'=0$ in the second relation and we 
obtain the five relations saying that $(M, \alpha _M)$ is an $A$-bimodule.
\end{proof}

Similarly to Theorem 4.5 in \cite{homquantum3}, one can prove the following result:
\begin{proposition}
Let $(A, \mu _A)$ be an associative algebra, $\alpha _A:A\rightarrow A$ an 
algebra endomorphism, $M$ an $A$-bimodule in the usual sense with actions $A\ot M\rightarrow M$, $a\ot m
\mapsto a\cdot m$ and $M\ot A\rightarrow M$, $m\ot a\mapsto m\cdot a$, 
and $\alpha _M:M\rightarrow M$ a linear map satisfying the conditions 
$\alpha _M(a\cdot m)=\alpha _A(a)\cdot \alpha _M(m)$ and $\alpha _M(m\cdot a)=
\alpha _M(m)\cdot \alpha _A(a)$  
for all $a\in A$, $m\in M$. 
Then $(M, \alpha _M)$ becomes a bimodule over the Hom-associative algebra 
$A_{\alpha _A}$, with actions  $A_{\alpha _A}\ot M\rightarrow M$, $a\ot m
\mapsto a\triangleright m:=\alpha _M(a\cdot m)=\alpha _A(a)\cdot \alpha _M(m)$ and 
$M\ot A_{\alpha _A}\rightarrow M$, $m\ot a\mapsto m\triangleleft a:=\alpha _M(m\cdot a)=
\alpha _M(m)\cdot \alpha _A(a)$. 
\end{proposition}
\begin{definition}
Let $(H, \mu _H, \Delta _H, \alpha _H)$ be a Hom-bialgebra. An {\em $H$-bimodule Hom-algebra} is a 
Hom-associative algebra $(D, \mu _D, \alpha _D)$ that is both a left and a right $H$-module Hom-algebra 
and such that $(D, \alpha _D)$ is an $H$-bimodule. 
\end{definition}

We can introduce now the Hom-analogue of the L-R-smash product.
\begin{theorem}
Let $(H, \mu _H, \Delta _H, \alpha _H)$ be a Hom-bialgebra, $(D, \mu _D, \alpha _D)$ an $H$-bimodule 
Hom-algebra, with actions denoted by $H\ot D\rightarrow D$, $h\ot d\mapsto h\cdot d$ and $D\ot H
\rightarrow D$, $d\ot h\mapsto d\cdot h$, and assume that the structure maps $\alpha _D$ and 
$\alpha _H$ are both bijective. Define the linear maps 
\begin{eqnarray*}
&&R:H\ot D\rightarrow D\ot H, \;\;\;R(h\ot d)=\alpha _H^{-2}(h_1)\cdot \alpha _D^{-1}(d)\ot 
\alpha _H^{-1}(h_2), \\
&&Q:D\ot H\rightarrow D\ot H, \;\;\;Q(d\ot h)=\alpha _D^{-1}(d)\cdot \alpha _H^{-2}(h_2)\ot 
\alpha _H^{-1}(h_1),
\end{eqnarray*}
for all $d\in D$, $h\in H$. Then we have a Hom-L-R-twisted tensor product $D\;_Q\ot _RH$, which will 
be denoted by $D\nat H$ (we denote $d\ot h:=d\nat h$ for $d\in D$, $h\in H$) and called the 
{\em Hom-L-R-smash product} of $D$ and $H$. The structure map of $D\nat H$ is $\alpha _D\ot 
\alpha _H$ and its multiplication is
\begin{eqnarray*}
&&(d\nat h)(d'\nat h')=[\alpha _D^{-1}(d)\cdot \alpha _H^{-2}(h'_2)][\alpha _H^{-2}(h_1)\cdot 
\alpha _D^{-1}(d')]\nat \alpha _H^{-1}(h_2h'_1).
\end{eqnarray*}
\end{theorem}
\begin{proof}
Note first that $R$ is exactly the Hom-twisting map defining the Hom-smash product $D\# H$, so 
$R$ satisfies the conditions (\ref{lrhom1}), (\ref{lrhom3}), (\ref{lrhom4}). With a proof similar 
to the one in \cite{mp2}, one can prove that the map $Q$ satisfies the conditions 
(\ref{lrhom2}), (\ref{lrhom5}), (\ref{lrhom6}).\\
\underline{Proof of  (\ref{lrhom7})}:
\begin{eqnarray*}
h_R\ot (d_R)_Q\ot h'_Q&=&\alpha _H^{-1}(h_2)\ot (\alpha _H^{-2}(h_1)\cdot \alpha _D^{-1}(d))_Q\ot h'_Q\\
&=&\alpha _H^{-1}(h_2)\ot \alpha _D^{-1}(\alpha _H^{-2}(h_1)
\cdot \alpha _D^{-1}(d))\cdot \alpha _H^{-2}(h'_2)\ot \alpha _H^{-1}(h'_1)\\
&\overset{(\ref{hommod1})}{=}&\alpha _H^{-1}(h_2)\ot (\alpha _H^{-3}(h_1)
\cdot \alpha _D^{-2}(d))\cdot \alpha _H^{-2}(h'_2)\ot \alpha _H^{-1}(h'_1)\\
&\overset{(\ref{hombimodule})}{=}&\alpha _H^{-1}(h_2)\ot \alpha _H^{-2}(h_1)
\cdot (\alpha _D^{-2}(d)\cdot \alpha _H^{-3}(h'_2))\ot \alpha _H^{-1}(h'_1)\\
&\overset{(\ref{righthommod1})}{=}&\alpha _H^{-1}(h_2)\ot \alpha _H^{-2}(h_1)
\cdot (\alpha _D^{-1}(\alpha _D^{-1}(d)\cdot \alpha _H^{-2}(h'_2)))\ot \alpha _H^{-1}(h'_1)\\
&=&h_R\ot (\alpha _D^{-1}(d)\cdot \alpha _H^{-2}(h'_2))_R\ot \alpha _H^{-1}(h'_1)\\
&=&h_R\ot (d_Q)_R\ot h'_Q.
\end{eqnarray*}
\underline{Proof of  (\ref{lrhom8})}:
\begin{eqnarray*}
d_R\ot (h_R)_Q\ot d'_Q&=&\alpha _H^{-2}(h_1)\cdot \alpha _D^{-1}(d)\ot 
\alpha _H^{-1}(h_2)_Q\ot d'_Q\\
&=&\alpha _H^{-2}(h_1)\cdot \alpha _D^{-1}(d)\ot 
\alpha _H^{-1}(\alpha _H^{-1}(h_2)_1)\ot \alpha _D^{-1}(d')\cdot \alpha _H^{-2}(\alpha _H^{-1}(h_2)_2)\\
&\overset{(\ref{hombia3})}{=}&\alpha _H^{-2}(h_1)\cdot \alpha _D^{-1}(d)\ot 
\alpha _H^{-2}((h_2)_1)\ot \alpha _D^{-1}(d')\cdot \alpha _H^{-3}((h_2)_2)\\
&\overset{(\ref{hombia1})}{=}&\alpha _H^{-3}((h_1)_1)\cdot \alpha _D^{-1}(d)\ot 
\alpha _H^{-2}((h_1)_2)\ot \alpha _D^{-1}(d')\cdot \alpha _H^{-2}(h_2)\\
&\overset{(\ref{hombia3})}{=}&\alpha _H^{-2}(\alpha _H^{-1}(h_1)_1)\cdot \alpha _D^{-1}(d)\ot 
\alpha _H^{-1}(\alpha _H^{-1}(h_1)_2)\ot \alpha _D^{-1}(d')\cdot \alpha _H^{-2}(h_2)\\
&=&d_R\ot \alpha _H^{-1}(h_1)_R \ot \alpha _D^{-1}(d')\cdot \alpha _H^{-2}(h_2)\\
&=&d_R\ot (h_Q)_R\ot d'_Q,
\end{eqnarray*}
finishing the proof.
\end{proof}

As a consequence of Proposition \ref{deformLRttp}, we immediately obtain the following result:
\begin{proposition} \label{twistinglrsmash}
Let $(H, \mu _H, \Delta _H)$ be a bialgebra and $(D, \mu _D)$ an $H$-bimodule algebra in the usual sense, 
with actions denoted by $H\ot D\rightarrow D$, $h\ot d\mapsto h\cdot d$ and $D\ot H
\rightarrow D$, $d\ot h\mapsto d\cdot h$. Let $\alpha _H:H\rightarrow H$ be a bialgebra endomorphism 
and $\alpha _D:D\rightarrow D$ an algebra endomorphism such that 
$\alpha _D(h\cdot d)=\alpha _H(h)\cdot \alpha _D(d)$ and $\alpha _D(d\cdot h)=\alpha _D(d)\cdot \alpha _H(h)$ 
for all $d\in D$, $h\in H$. If we consider the Hom-bialgebra $H_{\alpha _H}=(H, \alpha _H\circ \mu _H, 
\Delta _H\circ \alpha _H, \alpha _H)$ and the Hom-associative algebra $D_{\alpha _D}=(D, \alpha _D\circ \mu _D, 
\alpha _D)$, then $D_{\alpha _D}$ is an $H_{\alpha _H}$-bimodule Hom-algebra with actions 
$H_{\alpha _H}\ot D_{\alpha _D}\rightarrow D_{\alpha _D}$, $h\ot d\mapsto h\triangleright d:=
\alpha _D(h\cdot d)=\alpha _H(h)\cdot \alpha _D(d)$ and $D_{\alpha _D}\ot H_{\alpha _H}\rightarrow 
D_{\alpha _D}$, $d\ot h\mapsto d\triangleleft h:=\alpha _D(d\cdot h)=\alpha _D(d)\cdot \alpha _H(h)$. 
If we assume that moreover the maps $\alpha _H$ and $\alpha _D$ are bijective, if we denote by $D\nat H$ 
the L-R-smash product between $D$ and $H$, then $\alpha _D\ot \alpha _H$ is an algebra endomorphism 
of $D\nat H$ and the Hom-associative algebras $(D\nat H)_{\alpha _D\ot \alpha _H}$ and 
$D_{\alpha _D}\nat H_{\alpha _H}$ coincide. 
\end{proposition}
\begin{example} \label{Hstar}
Let $(H, \mu _H, \Delta _H, \alpha _H)$ be a Hom-bialgebra such that $\alpha _H$ is bijective. The vector space 
$H^*$ becomes a Hom-associative algebra with multiplication and structure map defined by: 
\begin{eqnarray*}
&&(f\bullet g)(h)=f(\alpha _H^{-2}(h_1))g(\alpha _H^{-2}(h_2)), \\
&&\beta :H^*\rightarrow H^*, \;\;\;\beta (f)(h)=f(\alpha _H^{-1}(h)), 
\end{eqnarray*}
for all $f, g\in H^*$ and $h\in H$. Then for any $p, q\in \mathbb{Z}$, this Hom-associative algebra $H^*$ 
can be organized as an $H$-bimodule Hom-algebra (denoted by $H^*_{p, q}$) with actions defined as follows:
\begin{eqnarray*}
&&\rightharpoonup :H\ot H^*\rightarrow H^*, \;\;\;
(h\rightharpoonup f)(h')=f(\alpha _H^{-2}(h')\alpha _H^p(h)), \\
&&\leftharpoonup :H^*\ot H\rightarrow H^*, \;\;\;(f\leftharpoonup h)(h')=f(\alpha _H^q(h)
\alpha _H^{-2}(h')),
\end{eqnarray*}
for all $h, h'\in H$ and $f\in H^*$. Note that $\beta $ is bijective and $\beta ^{-1}=\alpha _H^*$, the transpose 
of $\alpha _H$. So, we can consider the Hom-L-R-smash product $H^*_{p, q}\nat H$, whose 
structure map is $\beta \ot \alpha _H$ and whose multiplication is 
\begin{eqnarray*}
(f\nat h)(f'\nat h')=[\alpha _H^*(f)\leftharpoonup \alpha _H^{-2}(h'_2)]\bullet 
[\alpha _H^{-2}(h_1)\rightharpoonup \alpha _H^*(f')]\nat \alpha _H^{-1}(h_2h'_1). 
\end{eqnarray*}
\end{example}
\begin{example}
Let $(H, \mu _H, \Delta _H, \alpha _H)$ be a Hom-bialgebra, $(A, \mu _A, \alpha _A)$ a left 
$H$-module Hom-algebra and 
$(C, \mu _C, \alpha _C)$ a right $H$-module Hom-algebra, with actions denoted by 
$H\ot A\rightarrow A$, $h\ot a\mapsto h\cdot a$ and 
$C\ot H\rightarrow C$, $c\ot h\mapsto c\cdot h$. Define $D:=A\ot C$ as the tensor product 
Hom-associative algebra. Define the linear maps 
\begin{eqnarray*}
&&H\ot (A\ot C)\rightarrow A\ot C, \;\;\;h\cdot (a\ot c)=h\cdot a\ot \alpha _C(c), \\
&&(A\ot C)\ot H\rightarrow A\ot C, \;\;\;(a\ot c)\cdot h=\alpha _A(a)\ot c\cdot h. 
\end{eqnarray*}
Then one can easily check that $A\ot C$ with these actions becomes an $H$-bimodule Hom-algebra. 

Assume that moreover the structure maps 
$\alpha _H$, $\alpha _A$, $\alpha _C$ are bijective, so we can consider the Hom-L-R-smash product 
$(A\ot C)\nat H$. Then, by writing down the formula for the multiplication in $(A\ot C)\nat H$, 
and then by applying Proposition \ref{mis}, we obtain that $(A\ot C)\nat H$ is isomorphic to the 
two-sided Hom-smash product $A\# H\# C$. 
\end{example}
\section{Hom-diagonal crossed product}
\setcounter{equation}{0}
\begin{definition}
(i) If $(A, \mu , \alpha )$ is a Hom-associative algebra, we say that $A$ is {\em unital} if 
there exists an element $1_A\in A$ such that 
\begin{eqnarray}
&&\alpha (1_A)=1_A, \label{unit1}\\
&&1_Aa=a1_A=\alpha (a), \;\;\;\forall \;a\in A. \label{unit2}
\end{eqnarray}
If $f:A\rightarrow B$ is a morphism of Hom-associative algebras, we say that f is {\em unital} if 
$f(1_A)=1_B$.\\
(ii) If $(C, \Delta, \alpha )$ is a Hom-coassociative coalgebra, we say that $C$ is {\em counital} 
if there exists a linear map $\varepsilon _C:C\rightarrow k$ such that 
\begin{eqnarray}
&&\varepsilon _C\circ \alpha =\varepsilon _C, \label{counit1}\\
&&\varepsilon _C(c_1)c_2=c_1\varepsilon _C(c_2)=\alpha (c), \;\;\;\forall \;c\in C. \label{counit2}
\end{eqnarray}
If $f:C\rightarrow D$ is a morphism of Hom-coassociative coalgebras, we say that f is {\em counital} if 
$\varepsilon _D\circ f=\varepsilon _C$. 
\end{definition}
\begin{definition} (\cite{ms3}, \cite{ms4}) 
A {\em Hom-Hopf algebra} $(H, \mu _H, \Delta _H, \alpha _H, 1_H, \varepsilon _H, S)$ is a 
Hom-bialgebra such that $(H, \mu _H, \alpha _H, 1_H)$ is a unital Hom-associative algebra, 
$(H, \Delta _H, \alpha _H, \varepsilon _H)$ is a counital Hom-coassociative coalgebra, and $S:H\rightarrow H$ 
is a linear map (called the {\em antipode}) such that 
\begin{eqnarray}
&&\Delta _H(1_H)=1_H\ot 1_H, \\
&&\varepsilon _H(hh')=\varepsilon _H(h)\varepsilon _H(h'), \;\;\;\forall \;h, h'\in H, \\
&&\varepsilon _H(1_H)=1, \\
&&S(h_1)h_2=h_1S(h_2)=\varepsilon _H(h)1_H, \;\;\;\forall \;h\in H, \label{ant}\\
&&S\circ \alpha _H=\alpha _H\circ S. \label{suplant}
\end{eqnarray}
\end{definition}
As consequences of the axioms, and assuming that $\alpha _H$ is bijective, 
we have the following relations satisfied for a Hom-Hopf algebra: 
\begin{eqnarray}
&&S(1_H)=1_H, \label{antunit}\\
&&\varepsilon _H\circ S=\varepsilon _H, \\
&&S(hh')=S(h')S(h), \;\;\;\forall \;h, h'\in H, \label{antialg}\\
&&\Delta _H(S(h))=S(h_2)\ot S(h_1), \;\;\;\forall \;h\in H. 
\end{eqnarray}
If $S$ is bijective, we have, as an immediate consequence of (\ref{ant}), (\ref{antialg}) and (\ref{antunit}):
\begin{eqnarray}
&&S^{-1}(h_2)h_1=h_2S^{-1}(h_1)=\varepsilon _H(h)1_H, \;\;\;\forall \;h\in H. \label{invant}
\end{eqnarray}
\begin{proposition}\label{bijQ}
Let $(H, \mu _H, \Delta _H, \alpha _H, 1_H, \varepsilon _H, S)$ be a Hom-Hopf algebra with bijective 
antipode and $(D, \mu _D, \alpha _D)$ an $H$-bimodule Hom-algebra, with actions 
$H\ot D\rightarrow D$, $h\ot d\mapsto h\cdot d$ and $D\ot H
\rightarrow D$, $d\ot h\mapsto d\cdot h$, such that $\alpha _D$ and 
$\alpha _H$ are both bijective and 
\begin{eqnarray}
&&1_H\cdot d=d\cdot 1_H=\alpha _D(d), \;\;\;\forall \;d\in D. \label{unitaction}
\end{eqnarray}
Then the map $Q:D\ot H\rightarrow D\ot H$, $Q(d\ot h)=\alpha _D^{-1}(d)\cdot \alpha _H^{-2}(h_2)\ot 
\alpha _H^{-1}(h_1)$ is bijective, with inverse $Q^{-1}:D\ot H\rightarrow D\ot H$, 
$Q^{-1}(d\ot h)=\alpha _D^{-1}(d)\cdot \alpha _H^{-2}(S^{-1}(h_2))\ot 
\alpha _H^{-1}(h_1)$.
\end{proposition}
\begin{proof}
We check only that $Q^{-1}\circ Q=id$, the proof for $Q\circ Q^{-1}=id$ is similar and left to the reader. 
We compute:
\begin{eqnarray*}
(Q^{-1}\circ Q)(d\ot h)&=&
Q^{-1}(\alpha _D^{-1}(d)\cdot \alpha _H^{-2}(h_2)\ot 
\alpha _H^{-1}(h_1))\\
&=&\alpha _D^{-1}(\alpha _D^{-1}(d)\cdot \alpha _H^{-2}(h_2))\cdot \alpha _H^{-2}(
S^{-1}(\alpha _H^{-1}(h_1)_2))\ot \alpha _H^{-1}(\alpha _H^{-1}(h_1)_1)\\
&\overset{(\ref{hombia3}), \;(\ref{righthommod1})}{=}&
(\alpha _D^{-2}(d)\cdot \alpha _H^{-3}(h_2))\cdot \alpha _H^{-2}(
S^{-1}(\alpha _H^{-1}((h_1)_2)))\ot \alpha _H^{-2}((h_1)_1)\\
&\overset{(\ref{hombia1}), \;(\ref{suplant})}{=}&
(\alpha _D^{-2}(d)\cdot \alpha _H^{-4}((h_2)_2))\cdot \alpha _H^{-3}(
S^{-1}((h_2)_1))\ot \alpha _H^{-1}(h_1)\\
&\overset{(\ref{righthommod2})}{=}&\alpha _D^{-1}(d)\cdot (\alpha _H^{-4}((h_2)_2)\alpha _H^{-4}(
S^{-1}((h_2)_1)))\ot \alpha _H^{-1}(h_1)\\
&\overset{(\ref{invant})}{=}&\alpha _D^{-1}(d)\cdot (\alpha _H^{-4}(\varepsilon _H(h_2)1_H))
\ot \alpha _H^{-1}(h_1)\\
&\overset{(\ref{unit1})}{=}&\alpha _D^{-1}(d)\cdot 1_H\ot \alpha _H^{-1}(h_1\varepsilon _H(h_2))\\
&\overset{(\ref{unitaction}), \;(\ref{counit2})}{=}&d\ot h,
\end{eqnarray*}
finishing the proof.
\end{proof}

Assume now that we are in the hypotheses of Proposition \ref{bijQ} and consider the Hom-L-R-smash 
product $D\nat H=D\;_Q\ot _RH$. Since the map $Q$ is bijective, we can apply Proposition \ref{biject} 
and we obtain that the map $P:H\ot D\rightarrow D\ot H$, $P=Q^{-1}\circ R$ is a Hom-twisting map 
and we have an isomorphism of Hom-associative algebras $Q:D\ot _PH\simeq D\nat H$. 
\begin{definition}
This Hom-associative algebra $D\ot _PH$ will be denoted by $D\bowtie H$ and will be called the 
{\em Hom-diagonal crossed product} of $D$ and $H$. The map $P$ is defined by 
\begin{eqnarray*}
&&P(h\ot d)=(\alpha _H^{-3}(h_1)\cdot \alpha _D^{-2}(d))\cdot \alpha _H^{-3}(S^{-1}((h_2)_2))
\ot \alpha _H^{-2}((h_2)_1), 
\end{eqnarray*}
so the multiplication of $D\bowtie H$ is defined (denoting $d\ot h:=d\bowtie h$) by 
\begin{eqnarray*}
&&(d\bowtie h)(d'\bowtie h')=d[(\alpha _H^{-3}(h_1)\cdot \alpha _D^{-2}(d'))\cdot 
\alpha _H^{-3}(S^{-1}((h_2)_2))]
\bowtie \alpha _H^{-2}((h_2)_1)h'.
\end{eqnarray*}
\end{definition}

By a direct computation, one can check that the ''twisting principle'' holds also for diagonal crossed products, 
namely:
\begin{proposition}
Assume that we are in the hypotheses and notation of Proposition \ref{twistinglrsmash}, 
assuming moreover that $\alpha _H$ and $\alpha _D$ are bijective and $(H, \mu _H, \Delta _H, 1_H, 
\varepsilon _H)$ is a unital and counital Hopf algebra with bijective antipode $S$ and we have 
$\alpha _H(1_H)=1_H$, $\varepsilon _H\circ \alpha _H=\varepsilon _H$, $S\circ \alpha _H=\alpha _H\circ S$. 
Then $H_{\alpha _H}=(H, \alpha _H\circ \mu _H, 
\Delta _H\circ \alpha _H, \alpha _H, 1_H, \varepsilon _H, S)$ is a Hom-Hopf algebra, the 
$H_{\alpha _H}$-bimodule Hom-algebra $D_{\alpha _D}=(D, \alpha _D\circ \mu _D, \alpha _D)$ 
satisfies the hypotheses of Proposition \ref{bijQ}, the map $\alpha _D\ot \alpha _H$ is an 
algebra endomorphism of the diagonal crossed product $D\bowtie H$ and the Hom-associative 
algebras $(D\bowtie H)_{\alpha _D\ot \alpha _H}$ and $D_{\alpha _D}\bowtie H_{\alpha _H}$ coincide.
\end{proposition}

We need to characterize (left) modules over a Hom-diagonal crossed product, and we obtain first a 
characterization of (left) modules over a Hom-twisted tensor product. We begin with some definitions. 
\begin{definition}
Let $A\ot _RB$ be a Hom-twisted tensor product of the unital Hom-associative algebras 
$(A, \mu _A, \alpha _A, 1_A)$ and $(B, \mu _B, \alpha _B, 1_B)$. We say that $R$ is a {\em unital} 
Hom-twisting map if, for all $a\in A$, $b\in B$, we have $R(1_B\ot a)=a\ot 1_B$ and 
$R(b\ot 1_A)=1_A\ot b$. If this is the case, then $A\ot _RB$ is unital with unit $1_A\ot 1_B$, 
the maps $A\rightarrow A\ot _RB$, $a\mapsto a\ot 1_B$ and $B\rightarrow A\ot _RB$, $b\mapsto 1_A\ot b$ 
are unital morphisms of Hom-associative algebras and for all $a\in A$, $b\in B$ we have 
\begin{eqnarray}
&&(a\ot 1_B)(1_A\ot b)=\alpha _A(a)\ot \alpha _B(b). \label{abtwist}
\end{eqnarray}
\end{definition}
\begin{remark}
Let $D\bowtie H=D\ot _PH$ be a Hom-diagonal crossed product such that $D$ is a unital Hom-associative algebra 
and $h\cdot 1_D=1_D\cdot h=\varepsilon _H(h)1_D$, for all $h\in H$. Then $P$ is a unital Hom-twisting map, 
so $D\bowtie H$ is unital with unit $1_D\bowtie 1_H$. 
\end{remark}
\begin{remark}
If $H$ is a Hom-Hopf algebra with bijective antipode and such that $\alpha _H$ is bijective, we consider the 
$H$-bimodule Hom-algebra $H^*_{p, q}$ defined in Example \ref{Hstar}. It is easy to see that 
$H^*_{p, q}$ is unital with unit $\varepsilon _H$, its structure map $\beta $ is bijective, we have 
$1_H\rightharpoonup f=f\leftharpoonup 1_H=\beta (f)$, for all $f\in H^*$, and 
$h\rightharpoonup \varepsilon _H=\varepsilon _H\leftharpoonup h=\varepsilon _H(h)\varepsilon _H$, 
for all $h\in H$. Consequently, the Hom-diagonal crossed product $H^*_{p, q}\bowtie H$ is 
unital with unit $\varepsilon _H\bowtie 1_H$. 
\end{remark}
\begin{definition}
If $(A, \mu _A, \alpha _A, 1_A)$ is a unital Hom-associative algebra and $(M, \alpha _M)$ is a left 
$A$-module, we say that $M$ is {\em unital} if $1_A\cdot m=\alpha _M(m)$, $\forall \;m\in M$. If 
$\alpha _A$ is bijective, we 
denote by $_A{\mathcal M}$ the category whose objects are unital left $A$-modules $(M, \alpha _M)$ 
with $\alpha _M$ bijective, the 
morphisms being morphisms of left $A$-modules. 
\end{definition}
\begin{proposition}
Let $R:B\ot A\rightarrow A\ot B$ be a unital Hom-twisting map between the unital Hom-associative algebras 
$(A, \mu _A, \alpha _A, 1_A)$ and $(B, \mu _B, \alpha _B, 1_B)$. Let $M$ be a linear space and 
$\alpha _M:M\rightarrow M$ a linear map, and assume that $\alpha _A$, $\alpha _B$, $\alpha _M$ are bijective. 
Then $(M, \alpha _M)$ is a unital left $A\ot _RB$-module if and only if $(M, \alpha _M)$ is a unital left 
$A$-module and a unital left $B$-module (actions denoted by $\cdot $) satisfying the compatibility condition 
\begin{eqnarray}
&&\alpha _B(b)\cdot (a\cdot m)=\alpha _A(a_R)\cdot (b_R\cdot m), \;\;\;\forall \;a\in A, \;b\in B, \;m\in M. 
\label{compatib1}
\end{eqnarray}
If this is the case, the left $A\ot _RB$-module structure on $M$ is given by 
\begin{eqnarray}
&&(a\ot b)\cdot m=a\cdot (\alpha _B^{-1}(b)\cdot \alpha _M^{-1}(m)), \;\;\;\forall \;a\in A, \;b\in B, \;m\in M.
\label{defprod}
\end{eqnarray}
\end{proposition}
\begin{proof}
If $(M, \alpha _M)$ is a unital left $A\ot _RB$-module (with action denoted by $\cdot $), define actions 
of $A$ and $B$ on $M$ by $a\cdot m=(a\ot 1_B)\cdot m$ and $b\cdot m=(1_A\ot b)\cdot m$. 
Obviously we have $1_A\cdot m=1_B\cdot m=\alpha _M(m)$, and the conditions (\ref{hommod1}) and 
(\ref{hommod2}) for the actions of $A$ and $B$ follow immediately from the ones corresponding to the 
action of $A\ot _RB$. We need to prove (\ref{compatib1}). We compute: 
\begin{eqnarray*}
((1_A\ot b)(a\ot 1_B))\cdot \alpha _M(m)&=&(1_Aa_R\ot b_R1_B)\cdot \alpha _M(m)\\
&=&(\alpha _A(a_R)\ot \alpha _B(b_R))\cdot \alpha _M(m)\\
&\overset{(\ref{abtwist})}{=}&((a_R\ot 1_B)(1_A\ot b_R))\cdot \alpha _M(m).
\end{eqnarray*}
On the other hand, by using (\ref{hommod2}), we have 
\begin{eqnarray*}
((1_A\ot b)(a\ot 1_B))\cdot \alpha _M(m)&=&(1_A\ot \alpha _B(b))\cdot ((a\ot 1_B)\cdot m)\\
&=&\alpha _B(b)\cdot (a\cdot m), 
\end{eqnarray*}
\begin{eqnarray*}
((a_R\ot 1_B)(1_A\ot b_R))\cdot \alpha _M(m)&=&(\alpha _A(a_R)\ot 1_B)\cdot ((1_A\ot b_R)\cdot m)\\
&=&\alpha _A(a_R)\cdot (b_R\cdot m), 
\end{eqnarray*}
so we obtain $\alpha _B(b)\cdot (a\cdot m)=\alpha _A(a_R)\cdot (b_R\cdot m)$. Finally, to prove 
(\ref{defprod}), we compute:
\begin{eqnarray*}
a\cdot (\alpha _B^{-1}(b)\cdot \alpha _M^{-1}(m))&=&(a\ot 1_B)\cdot ((1_A\ot \alpha _B^{-1}(b))
\cdot \alpha _M^{-1}(m))\\
&\overset{(\ref{hommod2})}{=}&((\alpha _A^{-1}(a)\ot 1_B)(1_A\ot \alpha _B^{-1}(b)))\cdot m\\
&\overset{(\ref{abtwist})}{=}&(a\ot b)\cdot m.
\end{eqnarray*}
Conversely, assume that $(M, \alpha _M)$ is a unital left $A$-module and a unital left $B$-module and 
(\ref{compatib1}) holds, and define an action of $A\ot _RB$ on $M$ by 
$(a\ot b)\cdot m=a\cdot (\alpha _B^{-1}(b)\cdot \alpha _M^{-1}(m))$. We have 
$$(1_A\ot 1_B)\cdot m=
1_A\cdot (1_B\cdot \alpha _M^{-1}(m))=1_A\cdot m=\alpha _M(m).$$
We prove (\ref{hommod1}):
\begin{eqnarray*}
\alpha _M((a\ot b)\cdot m)&=&\alpha _M(a\cdot \alpha _M^{-1}(b\cdot m))\\
&\overset{(\ref{hommod1})}{=}&\alpha _A(a)\cdot (b\cdot m)\\
&=&\alpha _A(a)\cdot (\alpha _B^{-1}(\alpha _B(b))\cdot \alpha _M^{-1}(\alpha _M(m)))\\
&=&(\alpha _A(a)\ot \alpha _B(b))\cdot \alpha _M(m)\\
&=&(\alpha _A\ot \alpha _B)(a\ot b)\cdot \alpha _M(m). 
\end{eqnarray*}
Now we prove (\ref{hommod2}):
\begin{eqnarray*}
((a\ot b)(a'\ot b'))\cdot \alpha _M(m)&=&
(aa'_R\ot b_Rb')\cdot \alpha _M(m)\\
&=&(aa'_R)\cdot ([\alpha _B^{-1}(b_R)\alpha _B^{-1}(b')]\cdot m)\\
&\overset{(\ref{hommod2})}{=}&(aa'_R)\cdot (b_R\cdot (\alpha _B^{-1}(b')\cdot \alpha _M^{-1}(m)))\\
&\overset{(\ref{hommod2})}{=}&\alpha _A(a)\cdot (a'_R\cdot \alpha _M^{-1}(b_R\cdot 
(\alpha _B^{-1}(b')\cdot \alpha _M^{-1}(m))))\\
&\overset{(\ref{hommod1})}{=}&\alpha _A(a)\cdot (a'_R\cdot (\alpha _B^{-1}(b_R)\cdot 
(\alpha _B^{-2}(b')\cdot \alpha _M^{-2}(m))))\\
&\overset{(\ref{compatib1})}{=}&\alpha _A(a)\cdot (b\cdot (\alpha _A^{-1}(a')\cdot 
(\alpha _B^{-2}(b')\cdot \alpha _M^{-2}(m))))\\
&\overset{(\ref{hommod1})}{=}&\alpha _A(a)\cdot (b\cdot \alpha _M^{-1}(a'\cdot 
(\alpha _B^{-1}(b')\cdot \alpha _M^{-1}(m))))\\
&=&(\alpha _A(a)\ot \alpha _B(b))\cdot (a'\cdot 
(\alpha _B^{-1}(b')\cdot \alpha _M^{-1}(m)))\\
&=&(\alpha _A\ot \alpha _B)(a\ot b)\cdot ((a'\ot b')\cdot m),
\end{eqnarray*}
finishing the proof.
\end{proof}
\begin{corollary}\label{moddiag}
Let $D\bowtie H$ be a Hom-diagonal crossed product such that $D$ is unital and 
$h\cdot 1_D=1_D\cdot h=\varepsilon _H(h)1_D$, $\forall \;h\in H$. If $M$ is a linear space and 
$\alpha _M:M\rightarrow M$ a bijective linear map, then $(M, \alpha _M)$ is a unital left $D\bowtie H$-module 
if and only if $(M, \alpha _M)$ is a unital left $D$-module and a unital left $H$-module (actions denoted by 
$\cdot $) such that, $\forall \;h\in H, \;d\in D, \;m\in M$:
\begin{eqnarray}
&&\alpha _H(h)\cdot (d\cdot m)=[(\alpha _H^{-2}(h_1)\cdot \alpha _D^{-1}(d))\cdot \alpha _H^{-2}
(S^{-1}((h_2)_2))]\cdot (\alpha _H^{-2}((h_2)_1)\cdot m), 
\end{eqnarray}
and if this is the case then we have 
\begin{eqnarray}
&&(d\bowtie h)\cdot m=d\cdot (\alpha _H^{-1}(h)\cdot \alpha _M^{-1}(m)), 
\;\;\;\forall \;d\in D, \;h\in H, \;m\in M. \label{defdiagprod}
\end{eqnarray}
\end{corollary}
\section{Left-right Yetter-Drinfeld modules}
\setcounter{equation}{0}
\begin{definition} (\cite{mp2})  Let $(C, \Delta _C , \alpha _C)$ be a Hom-coassociative coalgebra, 
$M$ a linear space and $\alpha _M:M\rightarrow M$ a linear map. 
A {\em right $C$-comodule} structure on $(M, \alpha _M)$ consists of a linear map 
$\rho :M\rightarrow M\ot C$ satisfying the following conditions:
\begin{eqnarray}
&&(\alpha _M\ot \alpha _C)\circ \rho =\rho \circ \alpha _M,  
\label{rightcom1}\\
&&(\alpha _M\ot \Delta _C)\circ \rho =(\rho \ot \alpha _C)\circ \rho . 
\label{rightcom2}
\end{eqnarray} 
We usually denote $\rho (m)=m_{(0)}\ot m_{(1)}$. If $C$ is counital, then $(M, \alpha _M)$ is called 
{\em counital} if $\varepsilon _C(m_{(1)})m_{(0)}=\alpha _M(m)$, for all $m\in M$. 
If $(M, \alpha _M)$ and $(N, \alpha _N)$ are right $C$-comodules, a morphism of right $C$-comodules 
$f:M\rightarrow N$ is a linear map with $\alpha _N\circ f=f\circ \alpha _M$ and 
$f(m)_{(0)}\ot f(m)_{(1)}=f(m_{(0)})\ot m_{(1)}$, for all $m\in M$. 
\end{definition}

The concept of left-left Yetter-Drinfeld module over a Hom-bialgebra was introduced in \cite{mp1}. 
Similarly one can introduce left-right Yetter-Drinfeld modules. Since we will be interested here 
to work over Hom-Hopf algebras, we will impose unitality conditions and bijectivity of structure maps 
in the definition. 
\begin{definition}
Let $(H, \mu _H, \Delta _H, \alpha _H, 1_H, \varepsilon _H, S)$ be a Hom-Hopf algebra with bijective 
antipode and bijective $\alpha _H$. Let $M$ be a linear space and $\alpha _M:M\rightarrow M$ 
a bijective linear map. Then $(M, \alpha _M)$ is called a left-right Yetter-Drinfeld module over $H$ if 
$(M, \alpha _M)$ is a unital left $H$-module (action denoted by $\cdot $) and a counital right 
$H$-comodule (coaction denoted by $m\mapsto m_{(0)}\ot m_{(1)}\in M\ot H$) satisfying the following 
compatibility condition, for all $h\in H$, $m\in M$:
\begin{eqnarray}
&&\alpha _H(h_1)\cdot m_{(0)}\ot \alpha _H^2(h_2)\alpha _H(m_{(1)})=(h_2\cdot m)_{(0)}\ot 
(h_2\cdot m)_{(1)}\alpha _H^2(h_1). \label{YDlr}
\end{eqnarray}
We denote by $_H{\mathcal YD}^H$ the category whose objects are left-right Yetter-Drinfeld modules 
over $H$, morphisms being linear maps that are morphisms of left $H$-modules and 
right $H$-comodules. 
\end{definition}
\begin{remark}
Similarly to what happens for left-left Yetter-Drinfeld modules (see \cite{mp1}), left-right Yetter-Drinfeld 
modules over Hopf algebras become, via the ''twisting procedure'', left-right Yetter-Drinfeld modules 
over Hom-Hopf algebras. 
\end{remark}
\begin{remark}
Similarly to what happens for Hopf algebras, one can prove that condition (\ref{YDlr}) is equivalent to 
(for all $h\in H$, $m\in M$)
\begin{eqnarray}
&&(h\cdot m)_{(0)}\ot (h\cdot m)_{(1)}=\alpha _H^{-1}((h_2)_1)\cdot m_{(0)}\ot 
[\alpha _H^{-2}((h_2)_2)\alpha _H^{-1}(m_{(1)})]S^{-1}(h_1). \label{echivYDlr}
\end{eqnarray}
\end{remark}

Similarly to what we have proved in \cite{mp1} for left-left Yetter-Drinfeld modules, one can 
prove the following result:
\begin{proposition}
Let $(H, \mu _H, \Delta _H, \alpha _H, 1_H, \varepsilon _H, S)$ be a Hom-Hopf algebra with bijective 
antipode and bijective $\alpha _H$. \\
(i) If $(M, \alpha _M)$ and $(N, \alpha _N)$ are objects in $_H{\mathcal YD}^H$, then 
$(M\ot N, \alpha _M\ot \alpha _N)$ becomes an object in $_H{\mathcal YD}^H$ (denoted in 
what follows by $M\hot N$) with structures 
\begin{eqnarray*}
&&H\ot (M\ot N)\rightarrow M\ot N, \;\;\;h\ot (m\ot n)\mapsto h_1\cdot m\ot h_2\cdot n, \\
&&M\ot N\rightarrow (M\ot N)\ot H, \;\;\;m\ot n\mapsto (m_{(0)}\ot n_{(0)})\ot 
\alpha _H^{-2}(n_{(1)}m_{(1)}).
\end{eqnarray*}
(ii) $(k, id_k)$ is an object in $_H{\mathcal YD}^H$, with action and coaction defined by $h\cdot \lambda =
\varepsilon _H(h)\lambda $ and $\lambda _{(0)}\ot \lambda _{(1)}=\lambda \ot 1_H$, for all $\lambda \in k$. \\
(iii) $_H{\mathcal YD}^H$ is a braided monoidal category, with tensor product $\hot $, unit $(k, id_k)$, 
associativity constraints, unit constraints and braiding and its inverse defined (for all $(M, \alpha _M)$, 
$(N, \alpha _N)$, $(P, \alpha _P)$ in $_H{\mathcal YD}^H$ and 
$m\in M$, $n\in N$, $p\in P$, 
$\lambda \in k$) by 
\begin{eqnarray*}
&&a_{M, N, P}:(M\hot N)\hot P\rightarrow M\hot (N\hot P), \;\;\;
a_{M, N, P}((m\ot n)\ot p)=\alpha _M^{-1}(m)\ot (n\ot \alpha _P(p)), \\
&&l_M:k\hot M\rightarrow M, \;\;\;l_M(\lambda \ot m)=\lambda \alpha _M^{-1}(m), \\
&&r_M:M\hot k\rightarrow M, \;\;\;r_M(m\ot \lambda )=\lambda \alpha _M^{-1}(m), \\
&&c_{M, N}:M\hot N\rightarrow N\hot M, \;\;\;c_{M, N}(m\ot n)=\alpha _N^{-1}(n_{(0)})\ot 
\alpha _M^{-1}(\alpha _H^{-1}(n_{(1)})\cdot m), \\
&&c_{M, N}^{-1}:N\hot M\rightarrow M\hot N, \;\;\;c_{M, N}^{-1}(n\ot m)=
\alpha _M^{-1}(\alpha _H^{-1}(S(n_{(1)}))\cdot m)\ot \alpha _N^{-1}(n_{(0)}).
\end{eqnarray*}
\end{proposition}

The proof of the following result is straightforward and is left to the reader.
\begin{proposition}
Let $(H, \mu _H, \Delta _H, \alpha _H, 1_H, \varepsilon _H, S)$ be a Hom-Hopf algebra with bijective 
antipode and bijective $\alpha _H$. Consider the unital Hom-associative algebra $H^*$, with 
multiplication and structure map defined by 
\begin{eqnarray*}
&&(f\bullet g)(h)=f(\alpha _H^{-2}(h_1))g(\alpha _H^{-2}(h_2)), \;\;\;\forall \;f, g\in H^*, \;h\in H, \\
&&\beta :H^*\rightarrow H^*, \;\;\;\beta (f)(h)=f(\alpha _H^{-1}(h)), \;\;\;\forall \;f\in H^*, \;h\in H.
\end{eqnarray*}
(i) If $(M, \alpha _M)$ is a counital right $H$-comodule, with coaction $m\mapsto m_{(0)}\ot m_{(1)}$, 
then $(M, \alpha _M)$ becomes a unital left $H^*$-module, with action $f\cdot m=f(m_{(1)})m_{(0)}$, 
for all $f\in H^*$, $m\in M$. \\
(ii) Assume that $H$ is moreover finite dimensional. If $(M, \alpha _M)$ is a unital left $H^*$-module 
(action denoted by $\cdot $), then $(M, \alpha _M)$ becomes a counital right $H$-comodule, 
with coaction defined by $M\rightarrow M\ot H$, $m\mapsto \sum _ie^i\cdot m\ot e_i$, where 
$\{e_i\}$, $\{e^i\}$ is a pair of dual bases in $H$ and $H^*$ (of course, the coaction does not 
depend on the choice of the dual bases). 
\end{proposition}

Let again $(H, \mu _H, \Delta _H, \alpha _H, 1_H, \varepsilon _H, S)$ be a Hom-Hopf algebra with bijective 
antipode and bijective $\alpha _H$. From now on, we will denote by $H^*$ the unital $H$-bimodule 
Hom-algebra $H^*_{0, 0}$ (notation as in Example \ref{Hstar}), whose unit is $\varepsilon _H$, 
multiplication $\bullet $, structure map $\beta $ and $H$-actions  
\begin{eqnarray*}
&&\rightharpoonup :H\ot H^*\rightarrow H^*, \;\;\;
(h\rightharpoonup f)(h')=f(\alpha _H^{-2}(h')h), \\
&&\leftharpoonup :H^*\ot H\rightarrow H^*, \;\;\;(f\leftharpoonup h)(h')=
f(h\alpha _H^{-2}(h')).
\end{eqnarray*}

For $f\in H^*$ and $h\in H$, we will also denote $f(h)=\le f, h\ri $. 

To simplify notation in the proofs of the next results, we will use the following form of Sweedler-type 
notation: 
\begin{eqnarray*}
&&h_1\ot (h_2)_1\ot (h_2)_2=h_1\ot h_{21}\ot h_{22}, \\
&&h_1\ot ((h_2)_1)_1\ot ((h_2)_1)_2\ot (h_2)_2=h_1\ot h_{211}\ot h_{212}\ot h_{22}, \;\;\;etc...
\end{eqnarray*}
\begin{proposition}
We have a functor $F:\;_H{\mathcal YD}^H\rightarrow \;_{H^*\bowtie H}{\mathcal M}$, 
given by $F((M, \alpha _M))=(M, \alpha _M)$ at the linear level, with $H^*\bowtie H$-action defined by 
\begin{eqnarray*}
&&(f\bowtie h)\cdot m=\le f, (\alpha _H^{-1}(h)\cdot \alpha _M^{-1}(m))_{(1)}\ri 
(\alpha _H^{-1}(h)\cdot \alpha _M^{-1}(m))_{(0)}, 
\end{eqnarray*}
for all $f\in H^*$, $h\in H$, $m\in M$. On morphisms, $F$ acts as identity. 
\end{proposition}
\begin{proof}
Let $(M, \alpha _M)\in \;_H{\mathcal YD}^H$. Since $M$ is a unital right $H$-comodule, it becomes a unital 
left $H^*$-module. By Corollary \ref{moddiag}, the only thing we need to prove in order to have 
$(M, \alpha _M)$ a unital left $H^*\bowtie H$-module with the prescribed action is the compatibility 
condition 
\begin{eqnarray*}
\alpha _H(h)\cdot (f(m_{(1)})m_{(0)})&=&\le (\alpha _H^{-2}(h_1)\rightharpoonup \alpha _H^*(f))
\leftharpoonup \alpha _H^{-2}(S^{-1}(h_{22})), (\alpha _H^{-2}(h_{21})\cdot m)_{(1)}\ri \\
&&\;\;\;\;\;\;\;\;
(\alpha _H^{-2}(h_{21})\cdot m)_{(0)}, 
\end{eqnarray*}
for all $f\in H^*$, $h\in H$, $m\in M$. We compute the right hand side as follows:
\begin{eqnarray*}
RHS&=&\le \alpha _H^{-2}(h_1)\rightharpoonup \alpha _H^*(f), \alpha _H^{-2}(S^{-1}(h_{22}))
\alpha _H^{-2}((\alpha _H^{-2}(h_{21})\cdot m)_{(1)})\ri 
(\alpha _H^{-2}(h_{21})\cdot m)_{(0)}\\
&=&\le \alpha _H^*(f), \alpha _H^{-4}(S^{-1}(h_{22})(\alpha _H^{-2}(h_{21})\cdot m)_{(1)})
\alpha _H^{-2}(h_1)\ri (\alpha _H^{-2}(h_{21})\cdot m)_{(0)}\\
&=&\le f,  \alpha _H^{-3}(S^{-1}(h_{22})(\alpha _H^{-2}(h_{21})\cdot m)_{(1)})
\alpha _H^{-1}(h_1)\ri (\alpha _H^{-2}(h_{21})\cdot m)_{(0)}.
\end{eqnarray*}
By replacing $h$ with $\alpha _H^2(h)$, it turns out that we need to prove the following relation: 
\begin{eqnarray*}
&&\alpha _H^3(h)\cdot (f(m_{(1)})m_{(0)})=
\le f,  [\alpha _H^{-1}(S^{-1}(h_{22}))\alpha _H^{-3}((h_{21}\cdot m)_{(1)})]
\alpha _H(h_1)\ri (h_{21}\cdot m)_{(0)}.
\end{eqnarray*}
Note first that, by repeatedly applying (\ref{hombia1}), we obtain
\begin{eqnarray}
&&h_1\ot h_{211}\ot h_{2121}\ot h_{2122}\ot h_{22}=
h_1\ot \alpha _H(h_{21})\ot \alpha _H(h_{221})\ot h_{2221}\ot \alpha _H^{-2}(h_{2222}). 
\label{gaga1}
\end{eqnarray}
Now we can compute: \\[2mm]
${\;\;\;\;}$$\le f,  [\alpha _H^{-1}(S^{-1}(h_{22}))\alpha _H^{-3}((h_{21}\cdot m)_{(1)})]
\alpha _H(h_1)\ri (h_{21}\cdot m)_{(0)}$
\begin{eqnarray*}
&\overset{(\ref{echivYDlr})}{=}&\le f, \{\alpha _H^{-1}(S^{-1}(h_{22}))
[(\alpha _H^{-5}(h_{2122})\alpha _H^{-4}(m_{(1)}))\alpha _H^{-3}(S^{-1}(h_{211}))]\}
\alpha _H(h_1)\ri \\
&&\;\;\;\;\alpha _H^{-1}(h_{2121})\cdot m_{(0)}\\
&\overset{Hom-assoc.}{=}&\le f, \{[\alpha _H^{-2}(S^{-1}(h_{22}))
(\alpha _H^{-5}(h_{2122})\alpha _H^{-4}(m_{(1)}))]\alpha _H^{-2}(S^{-1}(h_{211}))\}
\alpha _H(h_1)\ri \\
&&\;\;\;\;\alpha _H^{-1}(h_{2121})\cdot m_{(0)}\\
&\overset{Hom-assoc.}{=}&\le f, \{[(\alpha _H^{-3}(S^{-1}(h_{22}))
\alpha _H^{-5}(h_{2122}))\alpha _H^{-3}(m_{(1)})]\alpha _H^{-2}(S^{-1}(h_{211}))\}
\alpha _H(h_1)\ri \\
&&\;\;\;\;\alpha _H^{-1}(h_{2121})\cdot m_{(0)}\\
&\overset{(\ref{gaga1})}{=}&\le f, \{[(\alpha _H^{-5}(S^{-1}(h_{2222}))
\alpha _H^{-5}(h_{2221}))\alpha _H^{-3}(m_{(1)})]\alpha _H^{-1}(S^{-1}(h_{21}))\}
\alpha _H(h_1)\ri \\
&&\;\;\;\;h_{221}\cdot m_{(0)}\\
&\overset{(\ref{invant})}{=}&\le f, \{[\varepsilon _H(h_{222})1_H
\alpha _H^{-3}(m_{(1)})]\alpha _H^{-1}(S^{-1}(h_{21}))\}
\alpha _H(h_1)\ri h_{221}\cdot m_{(0)}\\
&=&\le f, \{\alpha _H^{-2}(m_{(1)})\alpha _H^{-1}(S^{-1}(h_{21}))\}
\alpha _H(h_1)\ri \alpha _H(h_{22})\cdot m_{(0)}\\
&\overset{(\ref{hombia1})}{=}&\le f, \{\alpha _H^{-2}(m_{(1)})\alpha _H^{-1}(S^{-1}(h_{12}))\}
h_{11}\ri \alpha _H^2(h_2)\cdot m_{(0)}\\
&\overset{Hom-assoc.}{=}&\le f, \alpha _H^{-1}(m_{(1)})\{\alpha _H^{-1}(S^{-1}(h_{12}))
\alpha _H^{-1}(h_{11})\}\ri \alpha _H^2(h_2)\cdot m_{(0)}\\
&\overset{(\ref{invant})}{=}&\le f, \alpha _H^{-1}(m_{(1)})\varepsilon _H(h_1)1_H
\ri \alpha _H^2(h_2)\cdot m_{(0)}\\
&=&\le f, m_{(1)}\ri \alpha _H^3(h)\cdot m_{(0)}, 
\end{eqnarray*}
and this is exactly what we wanted to prove. The fact that morphisms in 
$_H{\mathcal YD}^H$ become morphisms in $_{H^*\bowtie H}{\mathcal M}$ is easy to prove 
and is left to the reader. 
\end{proof}
\begin{proposition}
If $H$ is finite dimensional, then we have a functor $\;G:\;_{H^*\bowtie H}{\mathcal M}
\rightarrow \;_H{\mathcal YD}^H$, given by $G((M, \alpha _M))=(M, \alpha _M)$ at the linear level, 
and $H$-action and $H$-coaction on $M$:
\begin{eqnarray*}
&&h\cdot m=(\varepsilon _H\bowtie h)\cdot m, \\
&&M\rightarrow M\ot H, \;\;\;m\mapsto (e^i\bowtie 1_H)\cdot m\ot e_i:=m_{(0)}\ot m_{(1)}, 
\end{eqnarray*}
where $\{e_i\}$, $\{e^i\}$ is a pair of dual bases in $H$ and $H^*$. On morphisms, $G$ acts as identity. 
\end{proposition}
\begin{proof}
It is obvious that, for $(M, \alpha _M)\in \;_{H^*\bowtie H}{\mathcal M}$,  $G(M)$ is a unital left $H$-module 
and a counital right $H$-comodule (the coaction is obtained from the left $H^*$-action, which in turn 
is obtained by restricting the $H^*\bowtie H$-action). We need to prove the Yetter-Drinfeld 
compatibility condition (\ref{YDlr}). Note first that by (\ref{hombia1}) we have (for all $h\in H$)
\begin{eqnarray}
&&h_{11}\ot h_{121}\ot h_{122}\ot h_2=\alpha _H(h_1)\ot \alpha _H(h_{21})\ot h_{221}\ot 
\alpha _H^{-2}(h_{222}). \label{gaga2}
\end{eqnarray} 
Note also that, by applying on an element in $H$ on the first tensor component, one can see that 
\begin{eqnarray}
&&\beta ((h\rightharpoonup \alpha _H^{*\;2}(e^i))\leftharpoonup g)\ot e_i=
\beta (e^i)\ot (g\alpha _H^{-2}(e_i))\alpha _H^2(h), \label{gaga3}
\end{eqnarray}
for all $h, g\in H$. Now we compute:\\[2mm]
${\;\;\;\;\;}$$\alpha _H(h_1)\cdot m_{(0)}\ot \alpha _H^2(h_2)\alpha _H(m_{(1)})$
\begin{eqnarray*}
&=&(\varepsilon _H\bowtie \alpha _H(h_1))\cdot ((e^i\bowtie 1_H)\cdot m)\ot \alpha _H^2(h_2)
\alpha _H(e_i)\\
&\overset{(\ref{hommod2})}{=}&((\varepsilon _H\bowtie h_1)(e^i\bowtie 1_H))
\cdot \alpha _M(m)\ot \alpha _H^2(h_2)
\alpha _H(e_i)\\
&=&\{\beta ((\alpha _H^{-3}(h_{11})\rightharpoonup \alpha _H^{*\;2}(e^i))\leftharpoonup 
\alpha _H^{-3}(S^{-1}(h_{122})))\bowtie \alpha _H^{-1}(h_{121})\}\cdot \alpha _M(m)\\
&&\;\;\;\;\;\ot 
\alpha _H^2(h_2)\alpha _H(e_i)\\
&\overset{(\ref{gaga3})}{=}&(\beta (e^i)\bowtie \alpha _H^{-1}(h_{121}))\cdot \alpha _M(m) 
\ot \alpha _H^2(h_2)\{[\alpha _H^{-2}(S^{-1}(h_{122}))\alpha _H^{-1}(e_i)]h_{11}\}\\
&\overset{Hom-assoc.}{=}&(\beta (e^i)\bowtie \alpha _H^{-1}(h_{121}))\cdot \alpha _M(m) 
\ot \{\alpha _H(h_2)[\alpha _H^{-2}(S^{-1}(h_{122}))\alpha _H^{-1}(e_i)]\}\alpha _H(h_{11})\\
&\overset{Hom-assoc.}{=}&(\beta (e^i)\bowtie \alpha _H^{-1}(h_{121}))\cdot \alpha _M(m) 
\ot \{[h_2\alpha _H^{-2}(S^{-1}(h_{122}))]e_i\}\alpha _H(h_{11})\\
&\overset{(\ref{gaga2})}{=}&(\beta (e^i)\bowtie h_{21})\cdot \alpha _M(m) 
\ot \{[\alpha _H^{-2}(h_{222})\alpha _H^{-2}(S^{-1}(h_{221}))]e_i\}\alpha _H^2(h_1)\\
&\overset{(\ref{invant})}{=}&(\beta (e^i)\bowtie h_{21})\cdot \alpha _M(m) 
\ot (\varepsilon _H(h_{22})1_He_i)\alpha _H^2(h_1)\\
&=&(\beta (e^i)\bowtie \alpha _H(h_2))\cdot \alpha _M(m) 
\ot \alpha _H(e_i)\alpha _H^2(h_1)\\
&\overset{(\ref{defdiagprod})}{=}&\beta (e^i)\cdot (h_2\cdot m)\ot \alpha _H(e_i)\alpha _H^2(h_1)\\
&=&(h_2\cdot m)_{(0)}\ot (h_2\cdot m)_{(1)}\alpha _H^2(h_1),
\end{eqnarray*}
where for the last equality we used the fact that $\{\alpha _H(e_i)\}$ and $\{\beta (e^i)\}$ is also 
a pair of dual bases. So indeed $M\in \;_H{\mathcal YD}^H$. We leave to the reader to prove that 
morphisms in $_{H^*\bowtie H}{\mathcal M}$ become morphisms in 
$_H{\mathcal YD}^H$.
\end{proof}

Since it is obvious that the functors $F$ and $G$ are inverse to each other, we obtain:
\begin{theorem} \label{YDdiag}
If $H$ is a finite dimensional Hom-Hopf algebra with bijective antipode and bijective structure map, 
the categories $_{H^*\bowtie H}{\mathcal M}$ and 
$_H{\mathcal YD}^H$ are isomorphic. 
\end{theorem}
\section{The Drinfeld double}
\setcounter{equation}{0}
${\;\;\;}$
We recall first a variation of a result in \cite{mp1}:
\begin{theorem} (\cite{mp1})
Let $(H, \mu _H, \Delta _H, \alpha _H, 1_H, \varepsilon _H, R)$ be a 
unital and counital quasitriangular Hom-bialgebra such that  
$\alpha _H$ is bijective and 
$(\alpha _H\ot \alpha _H)(R)=R$. Then $_H{\mathcal M}$ is a prebraided monoidal category, 
with tensor product defined as 
in Proposition \ref{tensprodmod}, unit $(k, id_k)$ with action $h\cdot \lambda =\varepsilon _H(h)\lambda $ 
for all $h\in H$, $\lambda \in k$, 
associativity constraints defined by the same formula as the ones of the category $_H{\mathcal YD}^H$, i.e. 
$a_{M, N, P}=\alpha _M^{-1}\ot id_N\ot \alpha _P$, for $M, N, P\in \;
_H{\mathcal M}$, and prebraiding defined by $c_{M, N}:M\ot N\rightarrow N\ot M$, 
$c_{M, N}(m\ot n)=\alpha _N^{-1}(R^2\cdot n)\ot \alpha _M^{-1}(R^1\cdot m)$, for all 
$M, N\in \;_H{\mathcal M}$.
\end{theorem}

Let now $(H, \mu _H, \Delta _H, \alpha _H, 1_H, \varepsilon _H, S)$ be a finite dimensional 
Hom-Hopf algebra with bijective antipode and bijective $\alpha _H$. We will construct the Drinfeld double 
$D(H)$ of $H$, which will be a quasitriangular Hom-Hopf algebra. 

As a Hom-associative algebra, $D(H)$ is the Hom-diagonal crossed product $H^*\bowtie H$. So, 
its unit is $\varepsilon _H\bowtie 1_H$, its structure map is $\beta \ot \alpha _H$ and its 
multiplication is defined by
\begin{eqnarray*}
&&(f\bowtie h)(f'\bowtie h')=f\bullet [(\alpha _H^{-3}(h_1)\rightharpoonup \alpha _H^{*\;2}(f'))
\leftharpoonup 
\alpha _H^{-3}(S^{-1}(h_{22}))]
\bowtie \alpha _H^{-2}(h_{21})h',
\end{eqnarray*}
for all $f, f'\in H^*$ and $h, h'\in H$, where $\beta =\alpha _H^{*\;-1}$ and 
\begin{eqnarray*}
&&(f\bullet g)(h)=f(\alpha _H^{-2}(h_1))g(\alpha _H^{-2}(h_2)),\\
&&\rightharpoonup :H\ot H^*\rightarrow H^*, \;\;\;
(h\rightharpoonup f)(h')=f(\alpha _H^{-2}(h')h), \\
&&\leftharpoonup :H^*\ot H\rightarrow H^*, \;\;\;(f\leftharpoonup h)(h')=
f(h\alpha _H^{-2}(h')).
\end{eqnarray*}

By Theorem \ref{YDdiag}, the category $_{D(H)}{\mathcal M}$ is isomorphic to
$_H{\mathcal YD}^H$, which is a braided monoidal category. We transfer the structure from 
$_H{\mathcal YD}^H$ to $_{D(H)}{\mathcal M}$ and then to $D(H)$. We obtain thus the 
following result:
\begin{theorem}
$D(H)$ is a quasitriangular Hom-Hopf algebra and we have an isomorphism of braided monoidal 
categories $_{D(H)}{\mathcal M}\simeq  \;_H{\mathcal YD}^H$. The structure of $D(H)$ is the following. 

Its counit is $\varepsilon (f\bowtie h)=f(1_H)\varepsilon _H(h)$, for all $f\in H^*$, $h\in H$. 

Its comultiplication is defined by 
\begin{eqnarray*}
&&\Delta :D(H)\rightarrow D(H)\ot D(H), \;\;\;\Delta (f\bowtie h)=(f_2\circ \alpha _H^{-2}\bowtie h_1)
\ot (f_1\circ \alpha _H^{-2}\bowtie h_2),
\end{eqnarray*}
where we denoted $\mu _H^*:H^*\rightarrow H^*\ot H^*$, the dual of $\mu _H$, defined by 
$\mu _H^*(f)=f_1\ot f_2$ if and only if $f(hh')=f_1(h)f_2(h')$, for all $h, h'\in H$. 

The quasitriangular structure is the element 
\begin{eqnarray*}
&&R=\sum _i(\varepsilon _H\bowtie \alpha _H^{-1}(e_i))\ot (e^i\bowtie 1_H)\in D(H)\ot D(H), 
\end{eqnarray*}
where $\{e_i\}$, $\{e^i\}$ is a pair of dual bases in $H$ and $H^*$. It satisfies the extra condition 
$((\beta \ot \alpha _H)\ot (\beta \ot \alpha _H))(R)=R$. 

The antipode of $D(H)$ is given by the formula 
\begin{eqnarray*}
&&S_{D(H)}(f\bowtie h)=(\varepsilon _H\bowtie S(\alpha _H^{-1}(h)))(f\circ \alpha _H\circ S^{-1}\bowtie 1_H), 
\;\;\;\forall \;f\in H^*, \;h\in H. 
\end{eqnarray*}
\end{theorem}
\begin{proof}
We leave most of the details to the reader. Let us note that in order to prove (\ref{homQT3}), one has 
to prove first that $\Delta ^{cop}(\varepsilon _H\bowtie h)R=R\Delta (\varepsilon _H\bowtie h)$ and 
$\Delta ^{cop}(f\bowtie 1_H)R=R\Delta (f\bowtie 1_H)$, for all $f\in H^*$, $h\in H$. Let us 
prove one of the two properties of the antipode, namely
\begin{eqnarray*}
&&(f\bowtie h)_1S_{D(H)}((f\bowtie h)_2)=f(1_H)\varepsilon _H(h)\varepsilon _H\bowtie 1_H. 
\end{eqnarray*}
Note that as a consequence of the Hom-associativity of $H$ we have 
\begin{eqnarray}
&&(ab)(cd)=\alpha _H(a)(\alpha _H^{-1}(bc)d),  \;\;\;\forall \; a, b, c,d\in H. \label{4elem}
\end{eqnarray}
Now we compute: \\[2mm]
${\;\;\;\;\;}$$(f\bowtie h)_1S_{D(H)}((f\bowtie h)_2)$
\begin{eqnarray*}
&=&(f_2\circ \alpha _H^{-2}\bowtie h_1)S_{D(H)}(f_1\circ \alpha _H^{-2}\bowtie h_2)\\
&=&[(f_2\circ \alpha _H^{-1}\bowtie 1_H)(\varepsilon _H\bowtie \alpha _H^{-1}(h_1))]
[(\varepsilon _H\bowtie S(\alpha _H^{-1}(h_2)))(f_1\circ \alpha _H^{-1}\circ S^{-1}\bowtie 1_H)]\\
&\overset{(\ref{4elem})}{=}&(f_2\circ \alpha _H^{-2}\bowtie 1_H)
[(\beta ^{-1}\ot \alpha _H^{-1})((\varepsilon _H\bowtie \alpha _H^{-1}(h_1))
(\varepsilon _H\bowtie S(\alpha _H^{-1}(h_2))))\\
&&\;\;\;\;\;(f_1\circ \alpha _H^{-1}\circ S^{-1}\bowtie 1_H)]\\
&=&(f_2\circ \alpha _H^{-2}\bowtie 1_H)
[(\beta ^{-1}\ot \alpha _H^{-1})(\varepsilon _H\bowtie \alpha _H^{-1}(h_1S(h_2)))
(f_1\circ \alpha _H^{-1}\circ S^{-1}\bowtie 1_H)]\\
&\overset{(\ref{ant})}{=}&\varepsilon _H(h)(f_2\circ \alpha _H^{-2}\bowtie 1_H)
(f_1\circ \alpha _H^{-2}\circ S^{-1}\bowtie 1_H)\\
&=&\varepsilon _H(h)((f_2\circ \alpha _H^{-2})\bullet 
(f_1\circ \alpha _H^{-2}\circ S^{-1})\bowtie 1_H)\\
&\overset{(\ref{invant})}{=}&f(1_H)\varepsilon _H(h)\varepsilon _H\bowtie 1_H,
\end{eqnarray*}
finishing the proof.
\end{proof}

\end{document}